\documentclass[]{amsart}

\usepackage{amsthm}
\usepackage{amsmath}
\usepackage{microtype}
\usepackage{mathtools}
\usepackage{amssymb}
\usepackage{enumerate}
\usepackage{overpic}
\usepackage{graphicx}
\usepackage[hidelinks,pagebackref,pdftex]{hyperref}
\usepackage{booktabs}
\usepackage{color}
\usepackage[dvipsnames]{xcolor}
\usepackage{import}
\usepackage{tikz-cd}

\usepackage[T1]{fontenc}
\usepackage{lmodern}

\AtBeginDocument{%
   \def\MR#1{}
}

\usepackage{marginnote}
\long\def\@savemarbox#1#2{\global\setbox#1\vtop{\hsize\marginparwidth 
  \@parboxrestore\tiny\raggedright #2}}
\renewcommand*{\backref}[1]{}
\renewcommand*{\backrefalt}[4]{
  \ifcase #1
  [No citations.]
  \or [#2]
  \else [#2]
  \fi }

\numberwithin{equation}{section}
\theoremstyle{plain}
\newtheorem{theorem}[equation]{Theorem}
\newtheorem{corollary}[equation]{Corollary}
\newtheorem{lemma}[equation]{Lemma}
\newtheorem{conjecture}[equation]{Conjecture}
\newtheorem{proposition}[equation]{Proposition}

\newtheorem{test}[equation]{Test}

\newtheorem*{namedtheorem}{\theoremname}
\newcommand{\theoremname}{testing}
\newenvironment{named}[1]{\renewcommand{\theoremname}{#1}\begin{namedtheorem}}{\end{namedtheorem}}

\theoremstyle{definition}
\newtheorem{definition}[equation]{Definition}
\newtheorem{remark}[equation]{Remark}
\newtheorem{question}[equation]{Question}

\newcommand{\refthm}[1]{Theorem~\ref{Thm:#1}}
\newcommand{\reflem}[1]{Lemma~\ref{Lem:#1}}
\newcommand{\refprop}[1]{Proposition~\ref{Prop:#1}}
\newcommand{\refcor}[1]{Corollary~\ref{Cor:#1}}
\newcommand{\refrem}[1]{Remark~\ref{Rem:#1}}

\newcommand{\refconj}[1]{Conjecture~\ref{Conj:#1}}
\newcommand{\refeqn}[1]{\eqref{Eqn:#1}}
\newcommand{\refitm}[1]{\eqref{Itm:#1}}
\newcommand{\refdef}[1]{Definition~\ref{Def:#1}}
\newcommand{\refsec}[1]{Section~\ref{Sec:#1}}
\newcommand{\reffig}[1]{Figure~\ref{Fig:#1}}

\newcommand{\reftest}[1]{Test~\ref{Test:#1}}
\newcommand{\refques}[1]{Question~\ref{Ques:#1}}

\newcommand{\HH}{{\mathbb{H}}}
\newcommand{\RR}{{\mathbb{R}}}
\newcommand{\ZZ}{{\mathbb{Z}}}
\newcommand{\NN}{{\mathbb{N}}}
\newcommand{\CC}{{\mathbb{C}}}
\newcommand{\QQ}{{\mathbb{Q}}}

\newcommand{\calC}{\mathcal{C}}
\newcommand{\calD}{\mathcal{D}}
\newcommand{\calE}{\mathcal{E}}
\newcommand{\calH}{\mathcal{H}}

\newcommand{\calO}{\mathcal{O}}
\newcommand{\calP}{\mathcal{P}}
\newcommand{\calS}{\mathcal{S}}
\newcommand{\calT}{\mathcal{T}}
\newcommand{\calU}{\mathcal{U}}
\newcommand{\calV}{\mathcal{V}}

\newcommand{\from}{\colon} 

\newcommand{\bdy}{\partial}

\newcommand{\quotient}[2]{{\raisebox{0.2em}{$#1$}
           \!\!\left/\raisebox{-0.2em}{\!$#2$}\right.}}

\newcommand{\vol}{\operatorname{vol}}

\newcommand{\area}{\operatorname{area}}
\newcommand{\len}{\operatorname{len}}

\newcommand{\systole}{{\operatorname{sys}}}
\newcommand{\sysmin}{{\operatorname{sysmin}}}
\newcommand{\lamhom}{{\lambda_{\rm hom}}}
\newcommand{\muhom}{{\mu_{\rm hom}}}

\newcommand{\haze}{{\operatorname{haze}}}
\newcommand{\ERROR}{{10^{-5}}}
\renewcommand{\ss}{{\mathbf{s}}}

\newcommand{\zhat}{{\hat z}}
\newcommand{\flin}{{ f_{\textup{lin}} }}
\newcommand{\fpl}{{ f_{\textup{pl}} }}
\newcommand{\lnew}{{\ell_{\textup{new}}}}

\DeclareMathOperator{\spn}{span}

\title[Excluding cosmetic surgeries]{Excluding cosmetic surgeries on hyperbolic 3-manifolds}

\author[D.~Futer]{David Futer}
\address[]{Department of Mathematics, Temple University,
Philadelphia, PA 19122, USA}
\email[]{dfuter@temple.edu}

\author[J.~Purcell]{Jessica S.~Purcell}
\address[]{School of Mathematics, Monash University, Clayton, VIC 3800, Australia }
\email[]{jessica.purcell@monash.edu}

\author[S.~Schleimer]{Saul Schleimer}
\address[]{Department of Mathematics, 
University of Warwick, Coventry CV4 7AL, UK}
\email[]{s.schleimer@warwick.ac.uk}

\makeatletter
\@namedef{subjclassname@2020}{%
  \textup{2020} Mathematics Subject Classification}
\makeatother
\thanks{\today}

\begin{document}

\begin{abstract}
This paper employs knot invariants and results from hyperbolic geometry to develop a practical procedure for checking the cosmetic surgery conjecture on any given one-cusped manifold. This procedure has been used to establish the following computational results. First, we verify that all knots up to 19 crossings, and all one-cusped 3-manifolds in the SnapPy census, do not admit any purely cosmetic surgeries. Second, we check that a hyperbolic knot with at most 15 crossings only admits chirally cosmetic surgeries when the knot itself is amphicheiral. Third, we enumerate all knots up to 13 crossings that share a common Dehn filling with the figure--$8$ knot. The code that verifies these results is publicly available on GitHub.
\end{abstract}

\maketitle

\section{Introduction}

This paper is concerned with the question of whether two distinct Dehn fillings of a given 3-manifold can ever produce the same result.
Consider a compact orientable 3-manifold $N$ whose boundary is a single torus, and a \emph{slope} $s$, or isotopy class of simple closed curves on $\bdy N$. The \emph{Dehn filling of $N$ along $s$}, denoted $N(s)$, is the 3-manifold obtained by attaching a solid torus $D^2\times S^1$ to $N$ by a homeomorphism of boundaries that sends a meridian curve $\bdy D^2 \times \{ * \}$ to $s \subset \bdy N$. We use the terms \emph{Dehn filling} and \emph{Dehn surgery} interchangeably.

\begin{definition}\label{Def:Cosmetic}
Let $N$ be a compact oriented 3-manifold whose boundary is a single torus. Let $s,s'$ be distinct slopes on $\bdy N$. We call $(s, s')$ 
a \emph{cosmetic surgery pair} if  there is a homeomorphism $\varphi \from N(s) \to N(s')$. 
The pair is called    \emph{chirally cosmetic} if $\varphi$ is orientation-reversing, and 
\emph{purely cosmetic} if $\varphi$ is orientation-preserving.
\end{definition}

There are many examples of chirally cosmetic surgeries where $N$ is Seifert fibered. See Bleiler--Hodgson--Weeks \cite{BleilerHodgsonWeeks} for a survey, and Ni--Wu \cite{NiWu:Cosmetic} for more examples.  Ichihara and Jong have constructed one example of a chirally cosmetic surgery pair where the parent manifold $N$ and the filled manifolds $N(s), N(s')$ are all hyperbolic \cite{IchiharaJong:CosmeticBanding}.
By contrast, no purely cosmetic surgeries are known apart from the case where $N$ is a solid torus. This has led Gordon \cite{Gordon:ICM} to propose

\begin{conjecture}[Cosmetic surgery conjecture]\label{Conj:Cosmetic}
Let $N$ be a compact, oriented 3-manifold such that $\bdy N$ is an incompressible torus. 
Then $N$ has no purely cosmetic surgeries.
\end{conjecture}

This problem appears as \cite[Conjecture 6.1]{Gordon:ICM}, as well as \cite[Problem 1.9(a)]{K3}.
See also \refconj{CosmeticMultiCusp} for a version with multiple boundary tori, and \refconj{ChiralKnotCosmetic} for a version that seeks to exclude chirally cosmetic surgeries on knots in $S^3$. This paper provides extensive experimental evidence for all three conjectures.

\subsection{Purely cosmetic surgeries on knots in $S^3$}

\refconj{Cosmetic} has received the most attention in the setting where $N$ is the exterior of a knot $K \subset S^3$. In this setting, certain knot invariants can severely limit the slopes $s,s'$ involved in a purely cosmetic surgery pair.  Invariants with this property include the Alexander polynomial (thanks to the work of Boyer and Lines~\cite{BoyerLines}),  knot Floer homology (thanks to the work of Ni--Wu~\cite{NiWu:Cosmetic} and Hanselman \cite{Hanselman:Cosmetic}), and the Jones polynomial (thanks to the work of Ichihara--Wu~\cite{IchiharaWu} and Detcherry~\cite{Detcherry:Cosmetic}). We refer to \refsec{KnotInvariants} for a detailed survey of these results, formulated as algorithmic tests that can be performed on a knot $K$.

Sometimes, these knot invariants can 
obstruct cosmetic surgeries on $N = S^3 - K$ altogether.
Notably, obstructions coming from Heegaard Floer homology, combined with other invariants, have been used to prove \refconj{Cosmetic} for several infinite families of knots. These families include knots of genus $1$ (Wang~\cite{Wang:GenusOne}), torus knots (Ni and Wu~\cite{NiWu:Cosmetic}), two-bridge knots (Ichihara, Jong, Mattman, and Saito~\cite{Ichihara-Jong-Mattman-Saito}), pretzel knots (Stipsicz and Szabo~\cite{Stipsicz-Szabo:Cosmetic}), alternating knots with at least 7 twist regions (Ichihara and Jong~\cite{IchiharaJong:AlternatingCosmetic}), connected sums (Tao~\cite{Tao:ConnectedSumCosmetic}), and nontrivial cables (Tao \cite{Tao:CablesCosmetic}). 

After this paper was first distributed, Daemi, Lidman, and Miller Eismeier used instanton Floer homology to prove that the Alexander polynomial obstructs cosmetic surgeries \cite{DLME:Cosmetic}; see \refthm{DLME} for details. It follows as a consequence that \refconj{Cosmetic} holds for all alternating knots, all fibered knots, and all knots of genus $g \neq 2$.

Knot invariants have also been used to verify \refconj{Cosmetic} on large data sets. The method is to first apply a test (such as one of the tests described in \refsec{KnotInvariants}) to reduce the search to a small subset of knots, and to a finite and explicitly given set of slopes on these knots. Then, direct computation (often using hyperbolic invariants) can distinguish the Dehn fillings along the remaining slopes.
In this vein, Hanselman used knot Floer homology (see Tests~\ref{Test:Tau}--\ref{Test:HanselmanFull}) to verify \refconj{Cosmetic} for all knots up to 16 crossings \cite{Hanselman:Cosmetic}. Detcherry used Jones polynomials (see Tests~\ref{Test:JonesDeriv}--\ref{Test:Detcherry}), with an assist from knot Floer homology, to  verify \refconj{Cosmetic} for all knots up to 17 crossings \cite{Detcherry:Cosmetic}.

In this paper, we extend the verification of \refconj{Cosmetic} to 19 crossings:

\begin{named}{\refthm{19CrossingKnots}}
The cosmetic surgery \refconj{Cosmetic} holds for all nontrivial knots up to 19 crossings.
\end{named}

The method of proof is still to compute obstructions coming from knot invariants, reducing the sample dramatically. One innovation, compared to prior work, is that direct computation of knot Floer homology is never needed 
for the knots in the set.  Instead, we rely heavily on the Alexander polynomial and the Turaev genus (see \reftest{GenusThickQuick}), to capture most of the obstructive information contained in knot Floer homology.
This speeds up computations dramatically, making it feasible to check the 352,152,252 prime knots up to 19 crossings.

A second innovation is that our work is encapsulated in user-friendly code that can be used to check \refconj{Cosmetic} for any knot in $S^3$, or for any hyperbolic manifold with a single cusp. See \refsec{CodeGuide} for more details.

\subsection{Cosmetic surgery conjecture for general 3-manifolds}

Moving beyond knots in $S^3$, let $\overline N$ be a compact 3-manifold with boundary consisting of tori, and let $N$ be its interior. Since most of our tools come from hyperbolic geometry, we restrict to the setting where $N$ is hyperbolic. In this setting, \refconj{Cosmetic} has the following analogue.

\begin{conjecture}[Hyperbolic cosmetic surgery conjecture]\label{Conj:CosmeticMultiCusp}
Let $N$ be a finite-volume hyperbolic 3-manifold with one or more cusps. 
Let $\ss$ and $\ss'$ be tuples of slopes on the cusps of $N$. If there is an orientation-preserving homeomorphism $\varphi \from N(\ss) \to N(\ss')$, and this manifold is hyperbolic, then $\varphi$ restricts (after an isotopy) to a homeomorphism $N \to N$ sending $\ss$ to $\ss'$.
\end{conjecture}

Compare Kirby \cite[Problem 1.81(B)]{Kirby:Problems} for a closely related statement, which is too strong to be true.
Indeed, restricting to purely cosmetic surgeries is necessary, even when $N$ has a single cusp. For instance, Dunfield  (personal communication) has observed that in the census manifold $N = \texttt{o9\_39009}$, the slopes $(-1, 3)$ and $(-3, 2)$ are chirally cosmetic, with the further property that the core curves of Dehn filling solid tori are isotopic to geodesics. Ichihara and Jong  \cite{IchiharaJong:CosmeticBanding} have also constructed a one-cusped hyperbolic manifold admitting chirally cosmetic surgeries. 

For a one-cusped hyperbolic manifold, Conjectures~\ref{Conj:Cosmetic} and~\ref{Conj:CosmeticMultiCusp} are nearly equivalent. In one direction, observe that the conclusion of  \refconj{Cosmetic} is stronger, hence \refconj{Cosmetic} implies \refconj{CosmeticMultiCusp}. In the other direction,
when $N$ is hyperbolic, we recall that every one-cusped manifold $N$  has a \emph{homological longitude}: a unique slope $\lamhom$ on the cusp that is trivial in $H_1(N, \QQ)$. See \refdef{HomologicalLongitude} for more details. Thus any orientation-preserving homeomorphism $\varphi \from N \to N$ must send $\lamhom$ to itself, hence must act trivially on the set of slopes on the cusp. Consequently,  \refconj{CosmeticMultiCusp} implies \refconj{Cosmetic} for all but the handful (no more than 10) of exceptional surgeries on $N$.

In this paper, we verify Conjectures~\ref{Conj:Cosmetic} and~\ref{Conj:CosmeticMultiCusp} for all one-cusped manifolds in the SnapPy census. 

\begin{named}{\refthm{CosmeticCensus}}
Conjectures~\ref{Conj:Cosmetic} and~\ref{Conj:CosmeticMultiCusp} hold for the 59,107 one-cusped manifolds in the SnapPy census.
\end{named}

The proof of \refthm{CosmeticCensus} makes heavy use of hyperbolic geometry. For each manifold $N$, we wish to restrict attention to a small, explicit list of slopes.
Thus we begin by proving \refconj{CosmeticMultiCusp} for all sufficiently long (tuples of) slopes, where ``long'' is explicitly quantified. For these long slopes, the only purely or chirally cosmetic surgeries come from symmetries of $N$ itself.
An ineffective version of this sort of statement appears in Bleiler--Hodgson--Weeks \cite{BleilerHodgsonWeeks}, and an effective version (with an explicit quantification) appears in our previous paper \cite[Theorem~7.29]{FPS:EffectiveBilipschitz}. 

Our effective statement requires the definition of \emph{normalized length}. 

\begin{definition}\label{Def:NormalizedLength}
Let $s$ be a slope on a horospherical  torus $T \subset N$. The \emph{normalized length} of $s$ is $L(s) = \len(s)/\sqrt{\area(T)}$. Observe that $L(s)$ is unchanged when the metric on $T$ is scaled, hence does not depend on the choice of cusp cross-section.

When $N$ has multiple cusps, let $\ss = (s_1, \ldots, s_n)$ be a tuple of slopes, with one slope per cusp. The 
\emph{normalized length} of $\ss$ is the quantity $L(\ss)$ satisfying $1/L(\ss)^2 = \sum_{i=1}^n 1/L(s_i)^2$. 
\end{definition}

We prove the following.

\begin{named}{\refthm{CosmeticImproved}}
Let $N$ be a cusped hyperbolic 3-manifold. Suppose that $\ss, \ss'$ are distinct tuples of slopes on the cusps of $N$, whose normalized lengths satisfy
\begin{equation}\label{Eqn:ImprovedHypoth}
L(\ss), L(\ss') \geq \max \left \{ 9.97, \, \sqrt{ \frac{2\pi}{\systole(N)}  + 56 } \right \}.
\end{equation}
Then any homeomorphism $\varphi \from N(\ss) \to N(\ss')$ restricts (after an isotopy) to a self-homeomorphism  of $N$ sending $\ss$ to $\ss'$. 
In particular, if $\systole(N) \geq 0.145$, then the above conclusion holds for all pairs $(\ss, \ss')$ of normalized length at least $9.97$.
\end{named}

\refthm{CosmeticImproved} is a slight strengthening of~\cite[Theorem 7.29]{FPS:EffectiveBilipschitz}. Compared to that statement, both the constant term and the variable term in the maximum of \refeqn{ImprovedHypoth} have become slightly smaller. This makes \refthm{CosmeticImproved} easier to apply in practice.

In \refsec{CosmeticProcedure}, we explain how \refthm{CosmeticImproved} leads to a practical procedure to check the Cosmetic surgery conjecture on one-cusped manifolds. Here is an outline. When $N$ has one cusp, there are only finitely many slopes $s$ whose normalized length is shorter than the bound in \refeqn{ImprovedHypoth}. Suppose for concreteness that $\systole(N) \geq 0.145$, hence the length bound becomes $9.97$. By a lemma of Agol~\cite[Lemma 8.2]{agol:6theorem}, there are at most $102$ slopes on the cusp of $N$ of normalized length at most $10$, hence $s$ must be one of these 102 slopes. If $N(s) \cong N(s')$ and this manifold is hyperbolic, then volume considerations imply that $L(s')$ is also bounded above; see \refthm{FinitenessOfPairs}. This leads to an explicit finite list of slope pairs $(s,s')$ to check, and every pair of filled manifolds $N(s), N(s')$ can be compared by computer. The finite list of pairs of exceptional surgeries on $N$ can also be compared by computer.

We have implemented this procedure in Sage, and the program is publicly available on GitHub \cite{FPS:GitHubCosmetic}. For any given one-cusped manifold $N$, the program forms the finite set of pairs of slopes that need to be checked, and then proceeds to compare them. \refsec{CosmeticProcedure} describes the procedure in detail. All of the comparisons between manifolds are done using rigorously verified hyperbolic invariants, or else using algebraic invariants such as the homology of covers. This includes the rigorous verification of $\systole(N)$, which is included in SnapPy versions 3.2 and following. The forthcoming work of Goerner, Haraway, Hoffman, and Trnkova \cite{GHHT:LengthSpectra} includes an explanation of the 
 verified length spectrum algorithm and a justification of the error bounds that are used in SnapPy 3.3.

As with all computer-assisted proofs, the correctness relies upon the rigor of our mathematical arguments (described in \refsec{Tools}), of our code (described in \refsec{CosmeticProcedure}), and of the software packages that we use (Sage, SnapPy, and Regina). 

\subsection{Chirally cosmetic surgeries}
There are no known counterexamples to \refconj{Cosmetic}, that is, no known examples of purely cosmetic surgeries on a 3-manifold $N$ with incompressible boundary
However, as stated above, there exist known examples of chirally cosmetic surgeries.
Recently Ichihara, Ito, and Saito proposed the following conjecture~\cite[Conjecture 1]{IIS:ChiralCasson}, which states that chirally cosmetic surgeries on knots in the 3-sphere only occur in very restricted ways. 

\begin{conjecture}[Chirally cosmetic surgery conjecture for knots in $S^3$]\label{Conj:ChiralKnotCosmetic}
Let $K$ be a knot in $S^3$ that is not a $(2,r)$ torus knot, and let $N = S^3 - K$. 
If $(s, s')$ is a chirally cosmetic pair of slopes on $\bdy N$, then $K$ has an orientation-reversing symmetry and $s' = - s$.
\end{conjecture}

See~\cite[Conjecture 1]{IIS:ChiralCasson}  and \cite[Propblem 1.9(d)]{K3} for a slightly stronger formulation: chirally cosmetic surgeries on knots in $S^3$ either satisfy \refconj{ChiralKnotCosmetic} or are particular pairs of slopes on $(2,r)$ torus knots. We restrict attention to non-torus knots because our methods are primarily hyperbolic. 

The statements of \refconj{ChiralKnotCosmetic} and of~\cite[Conjecture 1]{IIS:ChiralCasson} are carefully designed to get around known counterexamples. For instance, $(2,r)$ torus knots are known to admit chirally cosmetic surgeries;
see \cite[Corollary A.2]{IIS:ChirallyCosmetic}.
There are also known examples of one-cusped hyperbolic manifolds that do not embed in $S^3$ and admit chirally cosmetic surgeries, such as $N = \texttt{o9\_39009}$ and the example of Ichihara and Jong~\cite{IchiharaJong:CosmeticBanding}. 

\refconj{ChiralKnotCosmetic} is known to hold for several families of knots. These include genus one alternating knots \cite{IIS:ChiralCasson}, alternating odd pretzels \cite{Varvarezos:PretzelJKTR, Varvarezos:PretzelHFK}, positive two-bridge knots \cite{Ito:TwoBridge}, and special alternating knots with sufficiently many twist regions \cite{Ito:LargeTwist}. All of these families are alternating.

As with purely cosmetic surgeries, \refthm{CosmeticImproved} provides a useful practical tool for 
ruling out chirally cosmetic surgeries on a 
hyperbolic knot complement $N = S^3 - K$. Provided that $K$ is chiral (meaning, it does \emph{not} admit any orientation-reversing symmetries), \refthm{CosmeticImproved} constrains the slopes involved in chirally cosmetic surgeries to an explicit finite list. The procedure for checking this finite list, described in \refsec{ProcedureDescription}, is nearly identical to the one for purely cosmetic surgeries. Using this procedure, we have verified \refconj{ChiralKnotCosmetic} for hyperbolic knots up to 15 crossings.

\begin{named}{\refthm{ChiralKnots}}
  \refconj{ChiralKnotCosmetic}, the chirally cosmetic surgery conjecture for knots in $S^3$, holds for the 313,209 hyperbolic knots up to 15 crossings.
  \end{named}

In prior work, Ichihara, Ito, and Saito have verified \refconj{ChiralKnotCosmetic} for approximately 75\% of the 245 non-$(2,r)$-torus knots up to 10 crossings~\cite[Table 2]{IIS:ChiralCasson}. \refthm{ChiralKnots} expands this verification to a much larger data set.

\subsection{Common fillings of distinct cusped manifolds}
The same theoretical results and computational tools that can be used to detect common Dehn fillings of the same 3-manifold, as would occur in any counterexample to the cosmetic surgery conjecture,  can also be used to detect common Dehn fillings of distinct 3-manifolds. 

In \refsec{CommonFillings}, we walk through the adaptations needed to consider distinct parent 3-manifolds. We apply this in two results. In \refthm{CommonFillingLowCrossing}, we find all knots with at most 13 crossings that share a common nontrivial Dehn surgery with the figure-8 knot. In  \refthm{CommonFillingCensus}, we find all knot complements that can be triangulated with fewer than 10 tetrahedra that share a common nontrivial Dehn filling with the figure-8 knot complement. These results have been used by Kalfagianni and Melby~\cite{KalfagianniMelby} to construct many new $q$-hyperbolic knots in $S^3$. 

\subsection{Code on GitHub}\label{Sec:CodeGuide}
All of the code that was used to prove the computational theorems in this paper is publicly available on GitHub~\cite{FPS:GitHubCosmetic}.
We have endeavored to make the code user-friendly, extensively documented, and widely applicable.
Code that uses knot invariants is designed for arbitrary knots, and code for hyperbolic manifolds is designed for arbitrary one-cusped manifolds. Consequently, our programs should be able to handle examples outside the given data sets.

At the same time, we have not implemented a complete solution to the homeomorphism problem for 3-manifolds. Accordingly, our programs are not guaranteed to distinguish every non-homeomorphic pair of Dehn fillings of $N$. We have entered a number of invariants --- hyperbolic volume, length spectra, the homology of covers, Seifert invariants --- but we have necessarily had to stop somewhere. Indeed, the proof of \refthm{CosmeticCensus} contains several examples of Dehn fillings that had to be handled ``by hand.'' We emphasize that
all of the tests used by our programs are rigorous.

We point the reader to Remarks~\ref{Rem:CodeForKnotCensus}, \ref{Rem:CosmeticCensusCode}, \ref{Rem:AlmostAmphicheiral}, and \ref{Rem:CommonFillCode} for a description of the top-level routines that check the conjectures, as well as for some notes about running time.

\subsection{Acknowledgements}
As the computational and theoretical aspects of this project spanned several years, we have had occasion to lean on the input of many mathematicians. We thank 
Steve Boyer for suggesting that the Casson invariant works in settings more general than $S^3$.
We thank Jonathan Hanselman and Tye Lidman for explaining their work, including its many applications.
 We thank Renaud Detcherry, Kazuhiro Ichihara, Tetsuya Ito, and Toshio Saito for patiently answering questions about quantum invariants, especially in the setting of chirally cosmetic surgeries.  

We thank Matthew Stover for a crash course on arithmetic invariants. We also thank Craig Hodgson and James Clift for the computations described in \refprop{VolumeV3} and \refrem{Clift}.

We thank Nathan Dunfield and Matthias Goerner for answering numerous computational questions, and for implementing several SnapPy features (fast cover enumeration and verified length spectra) in a way that proved very timely for this project. 
We also thank Dunfield for writing code tying together SnapPy and Regina \cite{Dunfield:CensusFillings}. 
We thank Adam Lowrance for sharing his code for computing the diagram Turaev genus. 
We thank both Nathan Dunfield and Giles Gardam for suggesting the use of normal cores of subgroups (\refdef{HomologyInvt}) as a computable invariant of fundamental groups. 

We thank David Gabai, Effie Kalfagianni, and Joe Melby for independently asking about common fillings of distinct cusped manifolds, prompting us to work out the results of \refsec{CommonFillings}.

Futer was partially supported by NSF grants DMS--1907708 and DMS--2405046. Purcell was partially supported by ARC grant DP210103136. We thank Temple University and the Oberwolfach Mathematics Institute for their hospitality during collaborative visits. All of the computations described in this paper were carried out on a 2021 M1 iMac, purchased for this project using NSF grant DMS--1907708. 


\section{Obstructions and tests from knot invariants}\label{Sec:KnotInvariants}

In this section, we describe a suite of knot-theoretic tests that take a diagram $D(K)$ as input, compute one or more knot invariants, and (for certain ``generic'' values of the invariants) allow one to conclude that $S^3 - K$ has no purely cosmetic surgeries. These tests are convenient to perform in SnapPy, but do not use any hyperbolic geometry. We compare the effectiveness of these tests in subsection~\ref{Subsec:TestComparison}. Then, in \refthm{19CrossingKnots}, we apply these tests to verify \refconj{Cosmetic} for all knots up to 19 crossings.

When discussing knots in $S^3$, we follow the convention of denoting slopes on the boundary torus of $S^3 - K$ as rational numbers. That is, given the canonical meridian $\mu$ and homological longitude $\lambda$, we denote the slope $(p,q) = p\mu + q\lambda$ as the fraction $p/q \in \QQ \cup \{ \infty \}$.

\subsection{Theoretical tools}
Several of the tests below rely on a recent theorem of Hanselman \cite[Theorem 2]{Hanselman:Cosmetic}. In the following theorem, $g(K)$ is the Seifert genus of a knot $K$. Meanwhile, $th(K)$ is an invariant derived from the knot Floer homology $\widehat{HFK}(K)$. Recall that $\widehat{HFK}(K)$ is a bigraded, finitely generated  abelian group, whose nontrivial entries lie on finitely many lines of slope $1$ with respect to the bigrading. The \emph{thickness} $th(K)$ is the largest difference between the $y$--intercepts of these lines.

\begin{theorem}[Hanselman]\label{Thm:Hanselman}
Let $K \subset S^3$ be a nontrivial knot, with Seifert genus $g(K)$. Suppose that $s_1$ and $s_2$ are a purely cosmetic pair of slopes for $K$. Then one of the following holds:
\begin{enumerate}[$\: \: (1)$]
\item $\{s_1, s_2\} = \{\pm 2\}$ and $g(K) = 2$, or
\item $\{s_1, s_2\} = \{ \pm \tfrac{1}{q} \}$, where $1 \leq q \leq \dfrac{th(K) + 2g(K)}{2g(K)(g(K)-1)}$.
\end{enumerate}
In particular, if $th(K) < 2g(K)(g(K)-2)$, then $K$ cannot have purely cosmetic surgeries.
\end{theorem}

Much more recently, after this paper was first distributed, Daemi, Lidman, and Miller Eismeier proved a powerful theorem \cite[Theorems 1.2 and  6.1]{DLME:Cosmetic} that restricts the possibilities for cosmetic surgeries even further.

\begin{theorem}[Daemi, Lidman, Miller Eismeier]\label{Thm:DLME}
Let $K \subset S^3$ be a nontrivial knot. Then 
\begin{enumerate}[$\: \: (1)$]
\item The pair of slopes $\{ \pm \tfrac{1}{q} \}$ cannot be purely cosmetic for $K$.
\item If $\{\pm 2\}$ is a purely cosmetic pair for $K$, then the Alexander polynomial of $K$ is $\Delta_K = 1$.
\end{enumerate}
\end{theorem}

Combining Theorems~\ref{Thm:Hanselman} and~\ref{Thm:DLME}, we learn that any nontrivial knot in $S^3$ admitting a purely cosmetic pair of slopes $s_1, s_2$ must have $\Delta_K = 1$. This leads to \reftest{Alexander}, which is both powerful and easy to apply.
It also implies that \refconj{Cosmetic} holds for large families, including alternating knots and fibered knots. 

Even taken alone, \refthm{Hanselman} is a powerful result that inspired a considerable amount of subsequent work. Using it directly, as we do in \reftest{HanselmanFull}, is a costly computation because the number of generators in the chain complex for $\widehat{HFK}(K)$ is factorial in the crossing number $c(K)$. Fortunately, both the thickness $th(K)$ and the genus $g(K)$ can be estimated using simpler invariants.

\begin{definition}\label{Def:Turaev}
Given a knot diagram $D = D(K)$, the \emph{diagram Turaev genus} $g_T(D)$ is an invariant of $D$ that can be computed from a single Kauffman state of $D$. See \cite[Section 1.1]{ChampanerkarKofman:TuraevSurvey} for the exact definition.
The \emph{Turaev genus} of $K$, denoted $g_T(K)$, is a knot invariant defined by minimizing $g_T(D)$ over all diagrams of $K$. A knot $K$ is alternating if and only if $g_T(K) = 0$ \cite[Corollary 4.6]{DFKLS:graphs-on-surfaces}.

The important facts for our purposes are that $g_T(D)$ is quick to compute and that it satisfies
\[
th(K) \leq g_T(K) \leq g_T(D).
\]
Here, the first inequality is a theorem of Lowrance~\cite{Lowrance:TuraevThickness} and the second inequality holds by the definition of $g_T(K)$.
\end{definition}

\subsection{Practical tests}
The following tests are ordered roughly according to computational difficulty. The first three tests use the Alexander polynomial $\Delta_K(t)$, normalized to be symmetric in $t$ and $t^{-1}$.

\begin{test}[Alexander polynomial test]\label{Test:Alexander}
Given a knot $K \subset S^3$, compute the Alexander polynomial $\Delta_K$. If $\Delta_K \neq 1$, then Theorems~\ref{Thm:Hanselman} and~\ref{Thm:DLME} imply that $K$ cannot have purely cosmetic surgeries.
\end{test}

\begin{test}[Casson invariant test]\label{Test:Casson}
Given a 3-manifold $N = Y - K$, where $Y$ is an integer homology sphere, compute the (symmetrized) Alexander polynomial $\Delta_N = \Delta_K$ and the Casson invariant $\frac{1}{2} \Delta''_K(1)$. If $\Delta''_K(1) \neq 0$, then a theorem of Boyer and Lines \cite[Proposition 5.1]{BoyerLines} implies that $N$ cannot have purely cosmetic surgeries.
\end{test}

We remark that \reftest{Casson} is the only test  described in this section  that works outside the realm of knot complements in $S^3$. For this reason, it is not strictly subsumed by \reftest{Alexander}.

\begin{test}[Quick genus--thickness test]\label{Test:GenusThickQuick}
Given a knot $K \subset S^3$ with diagram $D = D(K)$, compute the diagram Turaev genus $g_T(D)$. 
Next, compute the Alexander polynomial $\Delta_K$, which provides a lower bound on the Seifert genus:
\[
\tfrac{1}{2} \spn \Delta_K  \leq g(K).
\]
If $g_T(D) < \spn \Delta_K (\tfrac{1}{2} \spn \Delta_K -2)$, then Lowrance's theorem~\cite{Lowrance:TuraevThickness} implies $th(K) < 2g(K)(g(K)-2)$,
hence Hanselman's \refthm{Hanselman} implies $K$ cannot have purely cosmetic surgeries.

If $g_T(D) = 0$, then $K$ is alternating \cite[Corollary 4.6]{DFKLS:graphs-on-surfaces}. If, in addition, $\spn \Delta_K = 2$, then $g(K) = 1$, hence Wang's theorem \cite{Wang:GenusOne} implies $K$ cannot have purely cosmetic surgeries.
\end{test}

We remark that for knots in $S^3$, \reftest{Casson} and \reftest{GenusThickQuick} are both strictly weaker than \reftest{Alexander}. However, since \reftest{Alexander} was not available during our initial work on this project, both  \reftest{Casson} and \reftest{GenusThickQuick} are used in our proof of \refthm{19CrossingKnots}.

The next two tests use the Jones polynomial $V_K(t)$.

\begin{test}[Jones derivative test]\label{Test:JonesDeriv}
Given a knot $K \subset S^3$, compute the Jones polynomial $V_K(t)$ and the third derivative $V_K'''$. If $V_K'''(1) \neq 0$, then a theorem of Ichihara and Wu~\cite{IchiharaWu}  implies that $K$ cannot have purely cosmetic surgeries.
\end{test}

\begin{test}[TQFT test]\label{Test:Detcherry}
Given a knot $K \subset S^3$, with diagram $D = D(K)$, compute the diagram Turaev genus $g_T(D)$. Next, compute the Jones polynomial $V_K(t)$ and its evaluation $V_K(e^{2\pi i/5})$. 
If $g_T(D) \leq 15$ and $V_K(e^{2\pi i/5}) \neq 1$, then Detcherry's theorem~\cite[Theorem 1.4]{Detcherry:Cosmetic}, combined with \refthm{Hanselman},  implies that $K$ cannot have purely cosmetic surgeries.
\footnote{If $V_K(e^{2\pi i/5}) \neq 1$, Detcherry's theorem \cite[Theorem 1.4]{Detcherry:Cosmetic} implies that any cosmetic surgery pair for $K$ is of the form $\{ \pm \frac{1}{5k} \}$. Thus \refthm{DLME} implies that  $K$ cannot have purely cosmetic surgeries, even without any information about $g_T(D)$. This makes \reftest{Detcherry} simpler to state and apply, but does not affect its power for small-crossing knots because all diagrams up to 31 crossings satisfy $g_T(D) \leq 15$.}
\end{test}

Detcherry's argument~\cite{Detcherry:Cosmetic} works by comparing Turaev--Viro invariants $TV_r(M)$, where $M$ is a Dehn filling of $K$ and $r \geq 5$ is a prime number. 
When $r=5$, the criterion can be expressed in terms of $V_K(e^{2\pi i/5})$, leading to a test that is both easy to state and easy to implement. We have not attempted to implement Detcherry's more general result~\cite[Theorem 1.5]{Detcherry:Cosmetic} for larger values of $r$, in part because \reftest{Detcherry} is already stunningly effective.

The final pair of tests uses invariants derived from the knot Floer homology $\widehat{HFK}(K)$.

\begin{test}[$\tau$ invariant test]\label{Test:Tau}
Given a knot $K \subset S^3$, compute the invariant $\tau(K)$ using $\widehat{HFK}(K)$. If $\tau(K) \neq 0$, then a theorem of Ni and Wu \cite{NiWu:Cosmetic} implies that $K$ cannot have purely cosmetic surgeries.
\end{test}

\begin{test}[HFK genus-thickness test]\label{Test:HanselmanFull}
Given a knot $K \subset S^3$, compute the Seifert genus $g(K)$ and the thickness $th(K)$ using $\widehat{HFK}(K)$. If $g(K) = 1$, then Wang's theorem \cite{Wang:GenusOne} implies that $K$ cannot have purely cosmetic surgeries.
If $th(K) < 2g(K)(g(K)-2)$, then Hanselman's \refthm{Hanselman} implies $K$ cannot have purely cosmetic surgeries.
\footnote{If  $g(K) \neq 2$, then \refthm{Hanselman} implies any cosmetic pair  for $K$  is of the form $\{\pm \frac{1}{q} \}$, hence \refthm{DLME} implies $K$ has no cosmetic surgeries at all. Thus, in light of \refthm{DLME}, it suffices to verify that $g(K) \neq 2$. This makes \reftest{HanselmanFull} simpler to state and to apply, but does not affect its power for small-crossing knots. 
 This is because all knots up to 17 crossings (and probably further) satisfy $th(K) \leq 5$, hence $g(K) \neq 2$ implies $K$ has no cosmetic surgeries by \cite[Corollary 3]{Hanselman:Cosmetic}.}
\end{test}

We remark that when $K$ is alternating,  \reftest{HanselmanFull} provides exactly the same information as \reftest{GenusThickQuick}: both tests rule out cosmetic surgeries whenever $g(K) \neq 2$.

\subsection{Effectiveness of the tests}\label{Subsec:TestComparison}
To assess the comparative effectiveness of the above tests, we ran each of them on the 9,755,328
prime knots up to 17 crossings. Running a single test on this sample took approximately 1.5 days of wall time for the Alexander--based tests, 4 days of wall time for the Jones--based tests, and 6 days of wall time for the Heegaard Floer--based tests. Here, the term \emph{wall time} refers to the clock time between starting the process (on one processor core) and the completion of the process.

Here is how the tests performed:

\smallskip
\begin{center}
\begin{tabular}{| c | c | c | c |}
\hline
Test & Knots eliminated & Knots remaining & Percent remaining \\
\hline
\reftest{Alexander}: Alexander polynomial & 9,752,895 & 2,433 & 0.025\% \\
\reftest{Casson}: Casson invariant & 8,745,275 & 1,010,053 & 10.3\% \\
\reftest{GenusThickQuick}: Quick genus-thickness & 9,668,010& 87,318 & 0.89\% \\
\reftest{JonesDeriv}: Jones derivative & 9,313,098 & 442,230 & 4.5\%  \\
\reftest{Detcherry}: TQFT & 9,755,231& 97 &  0.001\%  \\
\reftest{Tau}: $\tau$ invariant & 7,260,357 & 2,494,971 & 25.6\% \\
\reftest{HanselmanFull}: HFK genus-thickness & 9,740,390 & 14,938 & 0.15\% \\
\hline
\end{tabular}
\end{center}
\smallskip

%

One immediate conclusion is that \reftest{Detcherry} (TQFT) is extremely effective. Indeed, its effectiveness seems to grow as the crossing number increases: the proportion of prime knots with exactly $n$ crossings that satisfy $V_K(e^{2\pi i/5}) = 1$ is
\[
0.0035\% \text{ for } n=15, \quad 0.0017\% \text{ for } n=16, \quad 0.0006\% \text{ for } n=17.
\]
Since \reftest{Detcherry} is considerably slower than the Alexander--based tests, we note that the combination of Tests~\ref{Test:Casson}, \ref{Test:GenusThickQuick}, \ref{Test:JonesDeriv}, and~\ref{Test:Detcherry}  is very fast to run when performed in that order, and serves to rule out the vast majority of knots.
We also note that \reftest{Alexander} (Alexander polynomial), which only became available after our initial work on this project, is both very fast and very powerful all on its own.

The following table describes the level of dependence between the various tests. Given events $A$ and $B$, let $P(A |B)$ denote the conditional probability of $A$, given $B$. The ratio $\frac{P(A | B)}{P(A)}$ is called the \emph{normalized likelihood of $A$ given $B$}. 
Bayes' Theorem says that normalized likelihood is a symmetric quantity:
\[
\frac{P(A | B)}{P(A)} = \frac{P(B | A)}{P(B)}.
\]
In the following table, $B_i$ represents the event ``a random prime knot of up to 17 crossings survives after applying test $i$.'' Then, the $(i,j)$ entry of the table is the normalized likelihood $\frac{P(B_i | B_j)}{P(B_i)}$. Thus values less than $1$ suggest that tests $i$ and $j$ are negatively correlated, while values greater than $1$ suggest that they are positively correlated. Diagonal entries represent the inverse probability $ \frac{1}{P(B_i)}$ that a knot survives after test $i$; as expected, the diagonal entries for \reftest{Alexander}, \reftest{Detcherry}, and \reftest{HanselmanFull} are particularly large because those tests are particularly effective. 

\smallskip

\begin{center}
\begin{tabular}{| c | c | c | c | c | c | c | c | }
 \hline
 & \reftest{Alexander} & \reftest{Casson} & \reftest{GenusThickQuick} & \reftest{JonesDeriv} & \reftest{Detcherry} & \reftest{Tau} & \reftest{HanselmanFull} \\
 \hline 
 \reftest{Alexander} (Alexander)            & 4010 & 9.658	& 111.7 &	6.909 & 0 & 3.409 & 66.567 \\
 \reftest{Casson}  (Casson)               & & 9.658 & 1.334 & 3.851 & 1.394 & 1.413 & 0.952 \\
 \reftest{GenusThickQuick} (Quick)  & &  & 111.7 & 1.523 & 21.88 & 1.872 & 111.7 \\
 \reftest{JonesDeriv}  (Jones deriv)         & &  &  & 22.05 & 0.455	& 1.623 & 0.997 \\
\reftest{Detcherry}  (TQFT)            & &  &  &  & 100570 & 0.846 & 127.9 \\
\reftest{Tau}   ($\tau$ invariant)                     & &  &  &  &    & 3.910 & 1.787 \\
\reftest{HanselmanFull} (HFK genus)     & &  &  &  &  &  & 653.0 \\
\hline
\end{tabular}
\end{center}


\smallskip

The table contains a notable $0$, meaning that exactly \emph{zero} prime knots up to 17 crossings survive both \reftest{Alexander} (Alexander) and \reftest{Detcherry} (TQFT). However, this is not actually evidence of negative correlation: if \reftest{Alexander} and \reftest{Detcherry} were truly independent, we would expect $0.02$ knots to survive both, but of course the number of knots is an integer. Since none of the other entries in the table is substantially below $1$, we see that any negative correlations between the tests are very slight.

Some of the positive correlations in the table have a known cause.
For instance, \reftest{Alexander} (Alexander) is strongly positively correlated with  \reftest{GenusThickQuick} (quick genus--thickness) because any knot with $\Delta_K = 1$ must also survive all other  tests based on the Alexander polynomial. Similarly, 
 \reftest{GenusThickQuick} (quick genus--thickness) is strongly positively correlated with \reftest{HanselmanFull} (HFK genus-thickness) because any knot that survives \reftest{HanselmanFull} must also survive \reftest{GenusThickQuick}. The strong positive correlation between \reftest{Detcherry} (TQFT) and \reftest{HanselmanFull} (HFK genus-thickness) is more surprising, and merits further study.

\subsection{Results}
Using the above tests, we can prove the following. 

\begin{theorem}\label{Thm:19CrossingKnots}
The cosmetic surgery \refconj{Cosmetic} holds for all nontrivial knots up to 19 crossings.
\end{theorem}

\begin{proof}
Tao proved that \refconj{Cosmetic} holds for all composite knots \cite{Tao:ConnectedSumCosmetic}. Thus it suffices to check prime knots. 

Burton has compiled an enumeration of all 352,152,252 prime nontrivial knots up to 19 crossings~\cite{Burton:350MillionKnots}. We ran Tests~\ref{Test:Casson}--\ref{Test:Detcherry} on this list of knots. These tests rule out purely cosmetic surgeries on all but two of the knots: the knots $K_1$ and $K_2$ that are shown in Figure~\ref{Fig:Hard18Crossing}.

\begin{figure}
\begin{center}
\begin{overpic}[width=2in]{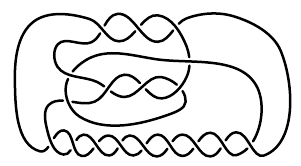}
\put(10,-8){DT:rarpjonmlkrqbafedcigh}
\put(93, 43){$K_1$}
\end{overpic}
\hspace{0.5in}
\begin{overpic}[width=2in]{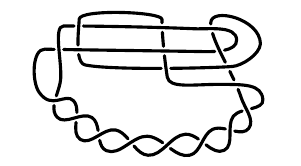}
\put(5,-8){DT:rarmlkjipNqedcbaOGRhF}
\put(88,43){$K_2$}
\end{overpic}
\vspace{0.1in}
\end{center}
\begin{caption} 
{The two prime knots up to 19 crossings that were not ruled out by Tests~\ref{Test:Casson}--\ref{Test:HanselmanFull}. Both exceptional knots have 18 crossings.}
\label{Fig:Hard18Crossing}
\end{caption}
\end{figure}

Now, \refthm{Hanselman} says that the only potential purely cosmetic pairs of surgery slopes on $K_i$ are $\{ \pm 1 \}$ and $\{ \pm 2 \}$. In each case, the filled manifolds are hyperbolic and are distinguished using verified hyperbolic volume.
\end{proof}

\begin{remark}\label{Rem:WhyTestsFail}
There are good theoretical reasons why Tests \ref{Test:Casson}--\ref{Test:HanselmanFull} all fail on the knots $K_1$ and $K_2$. Both of these knots  satisfy $g(K_i)=2$ and $\tau(K_i)=0$, so the $\widehat{HFK}$--based Tests~\ref{Test:Tau}--\ref{Test:HanselmanFull} do not provide any information.  
Since the two knots are \emph{thin} (meaning $th(K_i) = 0$) and have Alexander polynomials of the form appearing in \cite[Theorem 5]{Hanselman:Cosmetic}, Heegaard Floer invariants cannot distinguish the $\{ \pm 1 \}$ and $\{ \pm 2 \}$ Dehn surgeries on these knots. 
Consequently, \reftest{GenusThickQuick} (a quicker and weaker version of \reftest{HanselmanFull}) is similarly powerless.

In addition, performing $\pm 1/5$ surgery on a crossing circle about the large twist region of $K_i$ has the effect of replacing 9 half-twists by $-1$ half-twist, or 10 half-twists by 0. This surgery transforms each $K_i$ into the unknot, while leaving the value $J_K(e^{2\pi i / 5})$ unchanged. Compare \cite[Lemma 4.2]{Detcherry:Cosmetic}. Thus \reftest{Detcherry} cannot provide any information about $K_i$.

At the same time, the knots $K_1$ and $K_2$ have nontrivial Alexander polynomials, so \reftest{Alexander} excludes purely cosmetic surgeries on this pair. We did not use \reftest{Alexander} in the proof of \refthm{19CrossingKnots}, because \refthm{DLME} only became available after the first draft of this paper.
\end{remark}

\begin{remark}
  Running Tests~\ref{Test:Casson}--\ref{Test:Detcherry} on a database of 352 million knots took approximately 5 months of wall time (running the tests on two cores in parallel).

If one restricts to the hyperbolic and homological tools of \refsec{Tools} only, the search becomes far slower. Using \refthm{FinitenessOfPairs}, we were able to  check \refconj{Cosmetic} on the far smaller sample of 313,230 prime knots up to 15 crossings in approximately two weeks. This was part of the computation that yielded \refthm{ChiralKnots}.
\end{remark}

\begin{remark}\label{Rem:CodeForKnotCensus}
 \refthm{19CrossingKnots} was proved by running two python programs, available  on GitHub~\cite{FPS:GitHubCosmetic}. The first program,  {\tt prune\_using\_invariants}, performs the tests described above, allowing the user to turn specific tests on or off.
To prove \refthm{19CrossingKnots}, the program stepped through Burton's census and found that all knots except the two shown in \reffig{Hard18Crossing} were excluded.

The python program {\tt check\_knots} uses mostly hyperbolic methods (\refthm{FinitenessOfPairs}, combined with slope restrictions proved by Ni and Wu \cite{NiWu:Cosmetic}) to identify a finite number of slopes to check and then compare those slopes. This program was used on the remaining two knots. One notable feature of {\tt check\_knots} is that it works with a knot complement without needing access to a knot diagram. For example, it can quickly verify \refconj{Cosmetic} on the 1267 knot complements in the SnapPy census~\cite{snappy}, even though some of these knots have diagrams with hundreds of crossings.
\end{remark}

\section{Beyond the 3-sphere using hyperbolic tools}\label{Sec:Tools}

In this section, we discuss tools to test the cosmetic surgery conjecture, \refconj{Cosmetic}, on an arbitrary hyperbolic 3-manifold with one cusp. These tools have been written into a computer program, described in \refsec{CosmeticProcedure}. The tools are also used to test the chirally cosmetic surgery conjecture, \refconj{ChiralKnotCosmetic}, and to find common fillings of different surgery parents in \refsec{CommonFillings}.

The main philosophy behind our programs is to use geometry as much as possible. In particular, we compare closed hyperbolic manifolds using hyperbolic invariants such as the (complex) volume and the (complex) length spectrum. Rigorous bounds on which Dehn fillings should be compared require certain sharpened and simplified versions of theorems by Hodgson--Kerckhoff~\cite{hk:shape} and of a theorem from our previous work~\cite{FPS:EffectiveBilipschitz}. These results are described in Sections~\ref{Sec:Linearized}--\ref{Sec:ImprovedCosmetic}.

Non-hyperbolic manifolds are still handled from the perspective of geometrization. We compare closed, non-hyperbolic 3-manifolds by decomposing them along spheres and tori, as needed, into hyperbolic and Seifert fibered pieces. Hyperbolic pieces are compared as above, while Seifert fibered pieces are compared using their Seifert invariants.

\subsection{Invariants based on homology}
In addition to geometric tools, our program employs tools from algebraic topology to speed up computations and distinguish difficult cases.

Since the first homology  $H_1(M)$ of a 3-manifold $M$ is fast to compute, we sort the Dehn fillings of a cusped manifold $N$ into buckets that have the same first homology, and only compare pairs of manifolds within the same bucket. (See \reflem{Homology} and \S\ref{Step:SystoleShort}.) 
In addition to $H_1(M)$, we consider the first homology groups of small-degree covers. 

\begin{definition}\label{Def:HomologyOfCovers}
Let $M$ be a closed 3-manifold and $n$ a natural number. The 
\emph{cover homology invariant up to index $n$} is the set
\[
\calC(M,n) = \{ \big([\pi_1 M :H],  H^{ab} \big)  \mid [\pi_1 M : H] \leq n \}.
\]
In words, the invariant $\calC(M,n)$ packages together all of the homology groups of covers of $M$ up to degree $n$, while recalling the degree of the cover for each entry. 
\end{definition}

The recent versions of SnapPy (version 3.1 and later) can compute covers of $M$ extremely fast by employing multiple processor cores at the same time. This makes it feasible to compute $\calC(M,n) $ for $n \leq 7$, even on medium-sized examples.  In \refsec{CosmeticProcedure}, we use this invariant for both hyperbolic and non-hyperbolic manifolds. See \refrem{AlmostAmphicheiral} for more computational notes.

A slightly more sophisticated invariant is a package of homology groups of both small-index subgroups and their normal cores.

\begin{definition}\label{Def:HomologyInvt}
Given  a closed 3-manifold $M$, consider a finite-index subgroup $H < \pi_1 M$. The \emph{normal core $K = \operatorname{Core}(H)$} is the  intersection of all conjugates of $H$. Consider the $4$--tuple
\[
D(H) = \big( [\pi_1 M :H], [\pi_1 M: K], \, H^{ab}, \, K^{ab} \big)
\]
Then, for a natural number $n$, the \emph{core homology invariant up to index $n$} is the set
\[
\calD(M,n) = \{ D(H)  \mid [\pi_1 M : H] \leq n \}.
\]
\end{definition}

In practice, computing the normal core of a subgroup $H$ requires working in GAP~\cite{GAP}, which is slower than the cover enumeration in SnapPy. However, this computation is practical and the invariant $\calD(M,n)$ seems to be powerful.
In prior work, Dunfield successfully used this invariant (with $n=6$) to distinguish roughly 300,000 closed hyperbolic 3-manifolds~\cite[Proof of Theorem 1.4]{Dunfield:ComputationOrderability}.

Next, we consider the arrangement of homology groups of fillings in the Dehn surgery plane.

\begin{definition}\label{Def:HomologicalLongitude}
Let $N$ be a compact oriented 3-manifold with torus boundary. The \emph{homological longitude} of $N$ is a primitive homology class in $H_1(\bdy N, \ZZ)$ that is trivial in $H_1(N, \QQ)$. By the ``half lives, half dies'' lemma \cite[Lemma~3.5]{Hatcher:3Manifolds}, the homological longitude always exists and is unique up to sign.
\end{definition}

\begin{lemma}\label{Lem:Homology}
  Let $N$ be a compact oriented 3-manifold with torus boundary. Let $\lamhom$ be its homological longitude, and let $\muhom$ be any curve on $\bdy N$ whose intersection number with $\lamhom$ is $1$. Then there is an integer $k = k(N)$ such that the following hold for every relatively prime pair $(p,q)$, giving slope $s=p\muhom+q\lamhom$.
  
 \begin{itemize}
 \item The first Betti number of the Dehn filling $N(s)$, filled along the slope $s=p\muhom + q\lamhom$, satisfies
\[ b_1(N(s)) = 
 \begin{cases} b_1(N) & p = 0 \\
 b_1(N) -1 & p \neq 0. \end{cases}
 \]
 \item For all $p \neq 0$, the torsion of the first homology group of $N(s)$ satisfies
   $ | \operatorname{Tor} H_1 (N(s)) | = k |p| $.
 \end{itemize}
 \end{lemma}

\begin{proof}
Let  $i \from \bdy N \to N$ be the inclusion map, and let $i_*  \from H_1(\bdy M) \to H_1(M)$ be the induced homomorphism. By the ``half lives, half dies'' lemma  \cite[Lemma 3.5]{Hatcher:3Manifolds} and the definitions of $\lamhom$ and $\muhom$, we have that 
$i_*(\lamhom)$ is finite order and $i_*(\muhom)$ is infinite order in $H_1(N)$. Let $F$ be a maximal finite subgroup of $H_1(N)$ containing $i_*(\lamhom)$. Let $C = \langle \nu \rangle$ be a maximal cyclic subgroup of $H_1(N)$ containing $i_*(\muhom)$, so that $i_*(\muhom) = m \nu$. By the maximality of $C$ and $F$, we have a splitting
\[
H_1(N) \cong B \oplus C \oplus F,
\]
where $B \cong \ZZ^{b_1(N)-1}$ consists of non-peripheral homology classes, and $C,F$ are as above.

Now, let $s = p \muhom + q \lamhom$ be a Dehn filling slope. To compute the homology groups of $N(s)$, we first attach a $2$--handle $D^2$ along the slope $s$, and then attach the remaining 3-handle. The 3-handle has no effect on $H_1(N(s))$. The relevant portion of the Mayer--Vietoris sequence in reduced homology is
\[
\ldots \to H_1(S^1) \xrightarrow{ \: \alpha \: }  H_1(N) \oplus H_1(D^2) \xrightarrow{ \: \beta \: }  H_1(N \cup D^2) \xrightarrow{ \: \bdy \: } \widetilde H_0 (S^1) \to  \ldots
\]
Since $\widetilde H_0 (S^1) = 0$, we know that $\beta$ is onto. The kernel of $\beta$ is the image of $\alpha$, generated by $i_*(s) = i_*(p\muhom + q \lamhom)$.

If $p = 0$, then  $\ker(\beta) = \langle i_*(q \lamhom) \rangle \subset F$. Recall that $F$ is finite. Thus $b_1(N(s)) = b_1(N \cup D^2) = b_1(N)$, as claimed.

If $p \neq 0$, then $\ker(\beta) = \langle i_*(p \muhom + q \lamhom ) \rangle$, an infinite cyclic subgroup of $C \oplus F$. Thus $b_1(N(s)) = b_1(N \cup D^2) = b_1(N) -1$, as claimed. Recall that $i_*(p \muhom) = p m \nu$. Furthermore,
\begin{align*}
H_1 (N(s)) &\cong \quotient{H_1(N)}{\ker(\beta)} \\
& \cong \quotient{B \oplus C \oplus F}{\langle i_*(p \muhom + q \lamhom ) \rangle} \\
& \cong B \oplus \left[   \quotient{ \langle \nu \rangle \oplus F}{pm \nu = - q i_*(\lamhom )}   \right]
\end{align*}
In particular, the free part of $H_1(N(s))$ is isomorphic to $B \cong \ZZ^{b_1(N) -1}$. The torsion part of  $H_1(N(s))$ is the group in square brackets, whose order is $|pm| \cdot |F|$. In particular, the number $k$ in the statement of the lemma is $m |F|$.
\end{proof}

\begin{corollary}\label{Cor:HomLongitude}
  The homological longitude $\lamhom$ can never be part of a cosmetic surgery pair. \qed
\end{corollary}

\begin{corollary}\label{Cor:HomologyP=P'}
If the homology groups $H_1( N(p\muhom + q\lamhom))$ and $H_1( N(p'\muhom + q'\lamhom))$ are isomorphic, then $p = \pm p'$. \qed
\end{corollary}

\begin{remark}
The converse of \refcor{HomologyP=P'} can be false. Indeed, there exist cusped manifolds $N$ and Dehn fillings
$(p, q)$ and $(p, q')$ such that $H_1( N(p\muhom + q\lamhom)) \ncong H_1( N(p\muhom + q'\lamhom))$. 
For instance, the SnapPy census manifold $N = \tt{m172}$, equipped with the geometric framing, has homological longitude $\lamhom = (1,0)$. We can select $\muhom = (0,1)$. 
Then, taking $p=1$, we get
\begin{align*}
H_1( N(\muhom + q\lamhom)) \cong
\begin{cases}
\ZZ / 7 \oplus \ZZ/7, & q \equiv 3 \mod 7 \\
\ZZ / 49, & q \not\equiv 3 \mod 7.
\end{cases}
\end{align*}
%
\end{remark}

\subsection{Linearized Hodgson--Kerckhoff bounds}\label{Sec:Linearized}

The next two subsections are fairly technical. The main mathematical upshot is to establish certain estimates that constrain the set of slopes that need to be analyzed for excluding cosmetic surgeries. One of the key results, \refthm{HKSimplerBound}, constrains the set of Dehn fillings that satisfy an upper bound on volume. A second key result, \refthm{CosmeticImproved}, provides an improvement on \cite[Theorem~7.29]{FPS:EffectiveBilipschitz}. These results are combined in \refthm{FinitenessOfPairs}, which constrains the set of \emph{pairs} of slopes to check to a finite and explicitly computable set. 

The results below have an intrinsic intellectual value, because they are close to sharp. In addition, the  applications of these results in \refsec{ImprovedCosmetic} are quite useful: even small improvements on previous estimates provide valuable algorithmic speedups that matter a great deal when dealing with hundreds of thousands of examples. However, both the results and their proofs are technical in nature.
Accordingly, the casual reader who is mainly interested in the procedural and algorithmic aspects of this work is warmly invited to jump ahead to 
\refsec{CosmeticProcedure}.

In this subsection, we recall a theorem of Hodgson and Kerckhoff that bounds the change in volume under Dehn filling, as well as the length of the core geodesic.
 See \refthm{HKVolBound}.
The estimates produced by this result are rather sharp, but are not easy to apply because they involve inverting a function $f(z)$ that is defined via an integral; see \refdef{HKFunctions}. To address this practical shortcoming, we use secant line approximations to derive \refthm{HKSimplerBound}, a consequence of \refthm{HKVolBound} that is slightly weaker but easier to apply.

\begin{definition}\label{Def:Haze}
Fix constants $K = 3.3957$ and $z_0 = \sqrt{\sqrt{5}-2} =  0.4858 \ldots$. For $z \in [ z_0, 1 ]$, define a function
\[ \haze(z) = K \, \frac{z(1-z^2)}{1+z^2}. \]
By a derivative computation, 
 $\haze(z)$ is decreasing and invertible in this domain. Using Cardano's Formula, the inverse function $\haze^{-1}$ can be expressed as follows: 
\[
\haze^{-1}( K x) =  \frac{2 \sqrt{ x^2 + 3 }}{3}   \cos \left(   \frac{\pi}{3}  + \frac{1}{3} \tan^{-1} \! \left(\frac{-3  \sqrt{ -3x^4 - 33x^2 + 3 } }{ x^3 + 18x} \right)  \right) - \frac{x}{3}.
\]
Note that $\haze^{-1}(h)$ is defined and monotonically decreasing on $ [0, 1.0196]$. See \cite[Remark 4.23]{FPS:EffectiveBilipschitz}.
\end{definition}

\begin{definition}\label{Def:HKFunctions}
As above, let $K = 3.3957$ and $z_0 = \sqrt{ \sqrt{5} -2 }$. On the interval $[z_0, 1]$, define
\[
u(z) = K \, \frac{  z^2 (z^4 + 4z^2-1)}{2 \, (z^2 + 1)^3} 
\quad \:\: \text{and} \quad \:\:
v(z) = \frac{K}{4} \left( \frac{-2z^5+z^4-z^3+2z^2-z+1}{(z^2+1)^2} + \tan^{-1}(z) - \frac{\pi}{4} \right) \! .
\]
It is straightforward to check that  $\int_z^1 u(x) \, dx = v(z)$, hence $v'(z)=-u(z)$. Furthermore, $u(z) > 0$ on $(z_0,1)$, hence $v(z)$ is strictly decreasing and satisfies $\lim_{z \to 1} v(z) = 0$.
Next, define
\[
f(z) = K (1-z)  \exp \left(\int_z^1 \frac{-(x^4 +6x^2 +4x +1)}{(x+1)(x^2 + 1)^2} \,dx \right) \! .
\]
The function $f(z)$ is strictly decreasing (hence invertible) on its domain \cite[Lemma~5.4]{AtkinsonFuter:HighTorsion}.
 \end{definition}

Recall that the normalized length $L(\ss)$ of a tuple of slopes $\ss$ is defined in \refdef{NormalizedLength}.

\begin{theorem}[Hodgson--Kerckhoff~\cite{hk:shape}]\label{Thm:HKVolBound}
Let $N$ be a cusped hyperbolic 3-manifold. Let $\ss$ be a tuple of slopes on the cusps of $N$, whose total normalized length satisfies $L = L(\ss) \geq 7.5832$. 
Then the Dehn filling $M = N(\ss)$ is hyperbolic, and  the cores of the filling solid tori form a geodesic link $\Sigma \subset N$. Furthermore, setting  $\zhat = \zhat(\ss) = f^{-1} \Big( \big( \frac{2\pi}{L(\ss)} \big)^2 \Big)$, we have
\[
\vol(N) - \vol(M) \leq v(\zhat) =  \int_{\zhat}^1 u(z) \, dz
\qquad \text{and} \qquad
\len(\Sigma) \leq \frac{\haze(\zhat)}{2\pi}.
\]
\end{theorem}

\begin{proof}
The conclusion that $M = N(\ss)$ is hyperbolic and $\Sigma$ consists of geodesics is \cite[Theorem 1.1]{hk:shape}. The estimates on $\vol(M)$ and $\len(\Sigma)$ are a partial restatement of \cite[Theorem 5.12]{hk:shape}, with simplified notation. Hodgson and Kerckhoff formulated their theorem  in terms of functions
\[
H(z) = \frac{1}{\haze(z)}, \qquad G(z) = \frac{H(z)}{2} \frac{1-z^2}{z^2}
\]
that are defined on \cite[page 1079]{hk:shape}. The integrand in their upper bound on $\vol(N) - \vol(M)$ simplifies to our function $u(z)$. We remark that \cite[Theorem 5.12]{hk:shape} also includes lower bounds on $\vol(N) - \vol(M)$ and $\len(\Sigma)$, which we will not need.
\end{proof}

We wish to formulate a simpler version of \refthm{HKVolBound}, which replaces the hard-to-evaluate quantity $\zhat = \zhat(\ss) = f^{-1}  \Big( \big( \frac{2\pi}{L(\ss)} \big)^2 \Big)$ by a linear function of $y = 1/L(\ss)^2$. Toward that goal, we have the following lemma.

\begin{lemma}\label{Lem:Simpler_f}
Let $f(z)$ be the function of \refdef{HKFunctions}. Then $f(0.85) = 0.40104\ldots$ and $f(1) = 0$. Furthermore,
\begin{equation}\label{Eqn:f_linear}
f(z)  \geq
\begin{cases}
 \flin (z) = \frac{z - 1}{-0.374} & \text{if }  z \in [0.85,1] \\
 \fpl (z) = f(0.9) + \frac{z - 0.9}{-0.452}  & \text{if } z \in  [0.85,0.9].
 \end{cases}
\end{equation}
Consequently, for $y \in [0, f(0.85)]$, we have
\begin{equation}\label{Eqn:f_inverse_linear}
f^{-1}(y) \geq 
\begin{cases} 
\flin^{-1}(y) = 1 -0.374 y , \\
\fpl^{-1} (y) = \min\{ 1.031 -0.452y, \:  0.9 \}.
 \end{cases}
\end{equation}
\end{lemma}

See \reffig{Linearize} for a graph. The subscript in $\fpl^{-1}$ stands for ``piecewise linear.'' 

\begin{figure}
\begin{overpic}[width=4in]{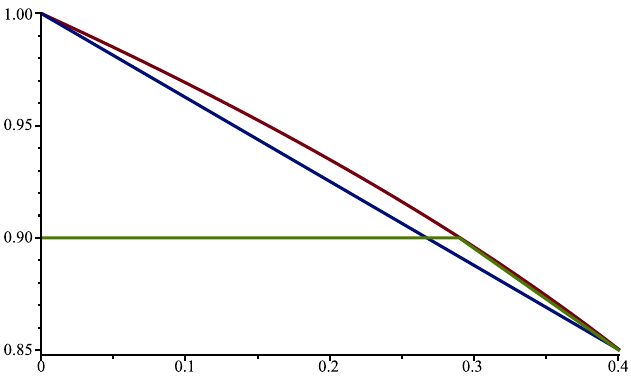}
\put(60,0){$y$}
\put(1,32){$z$}
\put(40,18){$\fpl^{-1}$}
\put(40,30){$\flin^{-1}$}
\put(40,43){$f^{-1}$}
\end{overpic}
\caption{The three functions appearing in \eqref{Eqn:f_inverse_linear}.}
\label{Fig:Linearize}
\end{figure}

\begin{proof}
The function $f(z)$ is decreasing on its domain by \cite[Lemma 5.4]{AtkinsonFuter:HighTorsion}, and concave down on $[0.83,1]$ by \cite[Lemma 5.9]{AtkinsonFuter:HighTorsion}. Consequently, for any sub-interval $[a,b] \subset [0.83,1]$, the graph of $f$ lies above the secant line segment connecting the points $(a,f(a))$ and $(b,f(b))$. Since $f$ is decreasing, the graph of $f$ also lies to the right of this secant line.

Direct computation shows that the slope of the secant line from $(0.85, f(0.85))$ to $(1,0)$ is lower (that is, more negative) than $-1/0.374$, hence the secant line between those points lies above the function $ f_{\text{lin}}(z)$, and the graph of $f(z)$ lies higher still. Similarly, direct computation shows that the slope of the secant line from $(0.85, f(0.85))$ to $(0.9, f(0.9))$ is lower (that is, more negative) than $-1/0.452$, hence $f(z)$ lies above $\fpl(z)$ on the sub-interval $[0.85, 0.9]$.

Solving the linear equation $\flin(z)=y$ for $z$ produces the linear function $\flin^{-1}(y) = 1 -0.374 y$. Since $f$ lies above and to the right of $\flin$, it follows that $f^{-1}(y) \geq \flin^{-1}(y)$ whenever $y \in [0, f(0.85)]$. Similarly, we solve the linear equation $\fpl(z)=y$ and find that $f^{-1}(y) \geq 1.031 - 0.452y  $ for $y \in [f(0.9), f(0.85)]$. Meanwhile, for $y \in [0, f(0.9)]$, we have $f^{-1}(y) \geq 0.9$ because $f$ is decreasing. Thus,  we learn that $f^{-1}(y) \geq \fpl^{-1} (y) =  \min\{ 1.031 -0.452y, \:  0.9 \}$, for all $y \in [0, f(0.85)]$. \end{proof}

Combining \refthm{HKVolBound} and \reflem{Simpler_f} with an estimate from \cite{FPS:EffectiveBilipschitz} yields the following theorem.

\begin{theorem}\label{Thm:HKSimplerBound}
Let $N$ be a cusped hyperbolic 3-manifold. Let $\ss$ be a tuple of slopes on the cusps of $N$, whose total normalized length satisfies $L = L(\ss) \geq 9.93$. 
Then the Dehn filling $M = N(\ss)$ is hyperbolic, and  the cores of the filling solid tori form a geodesic link $\Sigma \subset N$. Furthermore, 
\[
\vol(N) - \vol(N(\ss)) < v \! \left(  1 - \frac{14.77}{L(\ss)^2}   \right) .
\]
where $v(z)$ is as in \refdef{HKFunctions}. The length of $\Sigma$ satisfies
\begin{equation}\label{Eqn:Lnew}
\len(\Sigma) \leq \lnew(L) = \min \left\{  \frac{ \haze \Big( \! \min \big\{ 1.031 -  \tfrac{17.85}{L(\ss)^2} , \:  0.9 \big\} \Big) }{2\pi}, \:
 \frac{2\pi}{L^2 - 14.41} \right \}.
\end{equation}
\end{theorem}

\begin{proof}
The statement that $M = N(\ss)$ is hyperbolic and $\Sigma$ consists of geodesics is already included in \refthm{HKVolBound}.

Define $\zhat = \zhat(\ss) = f^{-1} \Big( \big( \frac{2\pi}{L(\ss)} \big)^2 \Big)$. 
The hypothesis $L(\ss) \geq 9.93$ implies $ \big( \frac{2\pi}{L(\ss)} \big)^2 < 0.401 < f(0.85)$, where the second inequality is by \reflem{Simpler_f}. Now, \reflem{Simpler_f} applies to give
\[
1 - \frac{14.77}{L(\ss)^2} \: < \: 1 -0.374 \left(  \frac{2\pi}{L(\ss)} \right)^2 \: \leq \: f^{-1} \left( \Big( \frac{2\pi}{L(\ss)} \Big)^{\! 2} \right)  \: = \:\zhat(\ss).
\]
 Consequently,  \refthm{HKVolBound} gives
\[
\vol(N) - \vol(N(\ss)) \: \leq \:  v(\zhat(\ss)) \: < \: v \! \left(  1 - \frac{14.77}{L(\ss)^2}   \right) .
\]
Here, the second inequality uses the fact that $v(z)$ is strictly decreasing (recall \refdef{HKFunctions}).

In a similar fashion, the inequality $f^{-1}(y) \geq \fpl^{-1}(y)$ in \reflem{Simpler_f} gives
\[
\min \left\{ 1.031 -  \frac{17.85}{L(\ss)^2} , \:  0.9 \right\} \: < \: \min \left\{ 1.031 -0.452  \left(  \frac{2\pi}{L(\ss)} \right)^{\! 2} \! , \:  0.9 \right\}
\: \leq \: f^{-1} \left( \Big( \frac{2\pi}{L(\ss)} \Big)^{\! 2} \right)  \: = \:\zhat(\ss).
\]
 Consequently,  \refthm{HKVolBound} gives
\[
\len(\Sigma) \: \leq \: \frac{\haze(\zhat)}{2\pi}
\: \leq \: \frac{ \haze \Big( \! \min \big\{ 1.031 -  \tfrac{17.85}{L(\ss)^2} , \:  0.9 \big\} \Big) }{2\pi}.
\]
Here, the first inequality uses \refthm{HKVolBound}. The second inequality uses the lower bound on $\zhat$ in previous displayed equation, combined with the fact that $\haze$ is decreasing. This proves half of the upper bound in \refeqn{Lnew}.

To complete the proof, it remains to show $\len(\Sigma) \leq \frac{2\pi}{L(\ss)^2 - 14.41}$. We consider two cases. When $9.93 \leq L(\ss) \leq 11.7$, we claim that
\begin{equation}\label{Eqn:LnewPart}
\len(\Sigma) \: \leq \: \frac{ \haze \Big( \! \min \big\{ 1.031 -  \tfrac{17.85}{L(\ss)^2} , \:  0.9 \big\} \Big) }{2\pi} \: < \:  \frac{2\pi}{L(\ss)^2 - 14.41}. 
\end{equation}
Here, the first inequality is the previous displayed equation. The second inequality is verified using interval arithmetic in Sage; see the ancillary files \cite{FPS:Ancillary} for full details.
Meanwhile, when $L = L(\ss) \geq 11.7$, we have $L^2 \geq 136.89$. Under this hypothesis,
the estimate $\len(\Sigma) \leq \frac{2\pi}{L(\ss)^2 - 14.41}$ is a consequence of \cite[Lemma 6.10]{FPS:EffectiveBilipschitz}; see the two displayed equations in the proof of \cite[Lemma 8.4]{FPS:EffectiveBilipschitz}.
\end{proof}

\begin{remark}\label{Rem:SecantLines}
The idea of using secant-line approximations to simplify the statement of \refthm{HKVolBound} is not new. It was previously used by Atkinson and Futer \cite{AtkinsonFuter:HighTorsion} to compare the volumes of different Dehn fillings of a cusped manifold, and by Haraway \cite{Haraway:ParentalTest} to develop a Dehn parental test.

The proof of the volume estimate in \refthm{HKSimplerBound} also applies to give the following simple length bound: 
\begin{equation}\label{Eqn:SimpleLenBound}
\len(\Sigma) \leq \frac{1}{2\pi} \haze  \left(  1 - \frac{14.77}{L(\ss)^2}   \right) .
\end{equation}
 While the length bound in \refeqn{Lnew} is more complicated to state than the one in \refeqn{SimpleLenBound}, it is somewhat sharper, while remaining straightforward to evaluate by computer for any given $L$. We will use this sharper estimate in the proof of \refthm{CosmeticImproved}.
\end{remark}

\subsection{Improved bound on cosmetic slopes}\label{Sec:ImprovedCosmetic}

Next, we prove strengthened versions of two results from \cite{FPS:EffectiveBilipschitz}. In \refthm{CosmeticImproved}, we give an effective upper bound on the normalized  length of a tuple of slopes involved in a cosmetic surgery. In \refthm{FinitenessOfPairs}, we restrict attention to one-cusped manifolds and prove an effective upper bound on the length of both slopes in a cosmetic pair.

\begin{theorem}\label{Thm:CosmeticImproved}
Let $N$ be a cusped hyperbolic 3-manifold. Suppose that $\ss, \ss'$ are distinct tuples of slopes on the cusps of $N$, whose normalized lengths  satisfy
\[
L(\ss), L(\ss') \geq \max \left \{ 9.97, \, \sqrt{ \frac{2\pi}{\systole(N)}  + 56 } \right \}.
\]
Then any homeomorphism $\varphi \from N(\ss) \to N(\ss')$ restricts (after an isotopy) to a self-homeomorphism  of $N$ sending $\ss$ to $\ss'$. 
In particular, if $\systole(N) \geq 0.145$, then the above conclusions hold for all pairs $(\ss, \ss')$ of normalized length at least $9.97$.
\end{theorem}

\refthm{CosmeticImproved} should be compared to \cite[Theorem~7.29]{FPS:EffectiveBilipschitz}. The statement of that previous result is nearly the same, except that the lower bound on length is $ \max \left \{ 10.1, \, \sqrt{ \frac{2\pi}{\systole(N)}  + 58 } \right \}$.
Thus, in \refthm{CosmeticImproved}, both the constant term and the variable term have become slightly smaller. In particular, when $\systole(N) \geq 0.145$, ruling out cosmetic surgeries on $N$ involves checking all slopes shorter than $9.97$ rather than $10.1$. The $\approx 1.3\%$ reduction in the normalized length of a slope leads to about a $\approx 2.6\%$ reduction in the number of slopes to check. Since the procedure of \refsec{CosmeticProcedure} compares each slope shorter than $9.97$ to some number of alternate slopes that may yield a manifold of the same volume, this results in a $\approx 4\%$ reduction in the number of \emph{pairs} of slopes to check. Accordingly, this strengthening is useful in applications.

The proof of \refthm{CosmeticImproved} closely parallels the proof of \cite[Theorem~7.29]{FPS:EffectiveBilipschitz}. We encourage the reader to go through the arguments of \cite[Section 7.3]{FPS:EffectiveBilipschitz}. While the arguments in that subsection rely on the cone deformation results developed earlier in \cite{FPS:EffectiveBilipschitz}, those results are only used as a black box.
Because the argument is so similar, we will mainly point out the few places that have to be modified. 

We begin with the following definition.

\begin{definition}\label{Def:SystoleL}
For $z \in (0,1)$ and $\ell \in (0, 0.5085)$, define a function
\begin{equation*}
F(z, \ell) =  \frac{(1+ z^2)}{  z^3 (3-z^2)} \cdot \frac{\ell}{10.667 - 20.977 \ell} \, .
\end{equation*}
Note that $F$ is positive everywhere on its domain, decreasing in $z$, and increasing in $\ell$.
Now, for any $L \geq 9.97$, define
\[
\sysmin_{\textup{new}}(L) = \lnew(L) \exp( 4\pi^2 F ( \haze^{-1} ( 4 \pi \, \lnew(L)  + 2\pi \, \ERROR), \, \lnew(L) ) ) ,
\]
where $\haze^{-1}$ is defined in \refdef{Haze} and $\lnew(L)$ is the function defined in \refeqn{Lnew}.
\end{definition}

\begin{lemma}\label{Lem:SysminNewProps}
The function $\sysmin_{\textup{new}}(L)$ is decreasing in $L$. Furthermore, for $L \geq 9.97$, we have
\[
\sysmin_{\textup{new}}(L) < \frac{2\pi}{L^2 - 56} .
\]
\end{lemma}

\begin{proof}
Recall from \refdef{Haze} and \cite[Remark 4.23]{FPS:EffectiveBilipschitz} that $\haze$ and $\haze^{-1}$ are both decreasing functions. It follows that $\lnew = \lnew(L)$ is a decreasing function of $L$, because  $\haze \left( 1.031 -   \tfrac{17.85}{L^2}  \right)$ and $\frac{2\pi}{L(\ss)^2 - 14.41}$ are both decreasing functions of $L$.
By a derivative computation,  the function
\[
\frac{F(z,\lnew)}{\lnew} = \frac{(1+ z^2)}{  z^3 (3-z^2)} \cdot \frac{1}{10.667 - 20.977 \lnew}
\]
 is decreasing in $z$ and increasing in $\lnew$. Thus
 $F ( \haze^{-1} ( 4 \pi \lnew +2\pi\ERROR), \, \lnew )$ is increasing in $\lnew$. 
 Since $\lnew = \lnew(L)$ is decreasing in $L$, we conclude that $\sysmin(L)$ is also decreasing.
 
 Next, we prove the inequality $\sysmin_{\textup{new}}(L) < \frac{2\pi}{L^2 - 56} $. For $L \in [9.97, 12]$, we verify this inequality using interval arithmetic in Sage \cite{FPS:Ancillary}. 
 For $L \geq 12$, direct computation shows that 
\[
\frac{ \haze \left( 1.031 -   \tfrac{17.85}{L^2}  \right) }{2\pi} <  \frac{\haze (0.9)}{2\pi} 
\qquad \text{and} \qquad
\frac{2\pi}{L^2 - 14.1} <  \frac{\haze (0.9)}{2\pi} ,
\]
hence we actually have $\lnew(L) = \frac{2\pi}{L^2 - 14.41}$ in \refdef{SystoleL}.

From here, the proof is extremely similar to the proof of \cite[Lemma 7.26]{FPS:EffectiveBilipschitz}. As in the proof of that proposition, we substitute $\lnew(L) = \frac{2\pi}{L^2 - 14.41}$ and apply some algebraic manipulation to show that the desired inequality $\sysmin_{\textup{new}}(L) < \frac{2\pi}{L^2 - 56} $ is equivalent to
\begin{equation}\label{Eqn:SUpper}
4\pi^2 \, \frac{F ( \haze^{-1} ( 4 \pi \lnew + 2\pi \ERROR), \, \lnew ) }{\lnew}
\: < \: \log \left ( \frac{L^2 - 14.41}{L^2 - 56}  \right)   \cdot \frac{L^2 - 14.41}{2\pi} \, . 
\end{equation}
Compare \cite[Equation (7.27)]{FPS:EffectiveBilipschitz}. We have already checked that the left-hand side of \refeqn{SUpper} is decreasing in $L$; meanwhile, the right-hand side is decreasing by a derivative calculation. As in the proof of \cite[Lemma 7.26]{FPS:EffectiveBilipschitz}, we show that as $L \to \infty$, the right-hand side limits to $\frac{56 - 14.41}{2\pi} = 6.619 \ldots$. Thus the right-hand side is greater than $6.6$ for all $L$. On the other hand, the left-hand side of \refeqn{SUpper} equals $5.614 \ldots$ when $L = 12$, hence is lower than $5.7$ for all $L \geq 12$.
\end{proof}

\begin{theorem}\label{Thm:UniqueShortest}
For a real number $L_0 \geq 9.97$, let $\sysmin_{\textup{new}}(L_0)$ be the function of \refdef{SystoleL}.
Let $N$ be a cusped hyperbolic 3-manifold whose systole is at least $\sysmin(L_0)$. Let $\ss$ be a tuple of surgery slopes on the cusps of $N$, whose normalized length is $L = L(\ss) \geq L_0$. 

Then the Dehn filled manifold $M = N(\ss)$ is hyperbolic. The core $\Sigma$ of the Dehn filling solid tori is isotopic to a geodesic link with an embedded tubular neighborhood of radius greater than $1$.
Finally, the only geodesics in $M$ of length at most $\ell = \len(\Sigma)$ are the components of $\Sigma$ itself.
\end{theorem}

\begin{proof}
This follows exactly the same argument as the proof of \cite[Theorem 7.28]{FPS:EffectiveBilipschitz}. That proof uses a function
\[
\ell_{\max}(L) = \frac{2 \pi}{L^2 - 16.03},
\]
and a function $\sysmin(L)$ that is defined exactly as in \refdef{SystoleL}, except using $\ell_{\max}(L)$ in place of $\lnew(L)$.
The salient facts in the proof are:
\begin{itemize}
\item In the filled manifold $M = N(\ss)$, the total length of $\Sigma$ satisfies $\len(\Sigma) \leq \ell_{\max}(L)$. By \refthm{HKSimplerBound}, we also have $\len(\Sigma) \leq \lnew(L)$.
\item In order to apply \cite[Theorem 7.19]{FPS:EffectiveBilipschitz} in the proof, one needs to check that $\len(\Sigma) \leq  0.0735$. Since $L \geq L_0 \geq 9.97$, \refthm{HKSimplerBound} gives the estimate $\len(\Sigma) \leq 0.07339$, which suffices.
\item The function $\sysmin(L)$ is decreasing in $L$. Similarly, $\sysmin_{\textup{new}}(L)$ is decreasing in $L$ by \reflem{SysminNewProps}.
\end{itemize}
The upshot is that once we substitute $\lnew(L)$ in place of $\ell_{\max}(L)$, the same three bullets are satisfied, and the rest of the argument goes through verbatim.
\end{proof}

\begin{proof}[Proof of \refthm{CosmeticImproved}]
This is a corollary of \refthm{UniqueShortest}, combined with Mostow rigidity and Gabai's topological rigidity theorem  \cite{Gabai:TopologicalRigidity}. See the proof of \cite[Theorem 7.29]{FPS:EffectiveBilipschitz} for full details.
\end{proof}

We also prove the following strengthened version of \cite[Corollary 1.13]{FPS:EffectiveBilipschitz}:

\begin{theorem}\label{Thm:FinitenessOfPairs}
Let $N$ be a one-cusped hyperbolic manifold. Suppose that slopes $s, s'$ are a cosmetic surgery pair for $N$. If the pair is chirally cosmetic, assume furthermore that there does not exist a symmetry $\varphi$ of $N$ with $\varphi(s) = s'$. Then, up to reordering $s$ and $s'$, the following hold: 
\begin{enumerate}[ \: \: $(1)$ ]
\item\label{Itm:S1short} $ L(s) < \max \left \{ 9.97, \, \sqrt{ \frac{2\pi}{\systole(N)}  + 56 } \right \} $.
\item\label{Itm:S2short} Either $L(s') < 9.97$ \: or \: $v \! \left( 1- \frac{14.77}{L(s')^2} \right) > \vol(N) - \vol(N(s))$.
\end{enumerate}
Here, $v(z)$ is the decreasing function of \refdef{HKFunctions}, and we are using the convention that the volume of a non-hyperbolic manifold is $0$.
\end{theorem}

We remark that an orientation--preserving symmetry $\varphi \from N \to N$ must preserve the homological longitude of $N$ (compare \refdef{HomologicalLongitude}), hence cannot send a slope $s$ to a distinct slope $s'$. This is why the theorem statement is only concerned with symmetries of $N$ in the chirally cosmetic setting.

\begin{proof}[Proof of \refthm{FinitenessOfPairs}]
Suppose without loss of generality that $L(s) \leq L(s')$. Since $s$ and $s'$ form a cosmetic surgery pair, and there does not exist a symmetry of $N$ with $\varphi(s) = s'$, \refthm{CosmeticImproved} says the shorter slope $s$ must satisfy \refitm{S1short}.

Now, consider slope $s'$. If $L(s') < 9.97$, then conclusion \refitm{S2short} holds. Alternately, if $L(s') \geq 9.97$, then \refthm{HKSimplerBound} implies $N(s')$ is hyperbolic, hence $N(s)$ is hyperbolic as well. \refthm{HKSimplerBound} also implies that
\[
\vol(N) - \vol(N(s)) \: = \: \vol(N) - \vol(N(s')) \: < \: v \! \left(  1 - \frac{14.77}{L(s')^2}   \right) ,
\]
hence conclusion \refitm{S2short} holds. 
\end{proof}

\refthm{FinitenessOfPairs} says  a cosmetic surgery pair on a one-cusped manifold must come from an explicit, finite list. Indeed, there are only finitely many slopes $s$ satisfying \refitm{S1short}. There are also finitely many slopes $s'$ satisfying $L(s') < 9.97$. For each $s$, a theorem of Gromov and Thurston \cite[Theorem 6.5.6]{thurston:notes} says that $\vol(N) - \vol(N(s)) > 0$. Thus $v \big( 1- \frac{14.77}{L(s')^2} \big)$ is bounded below by a strictly positive quantity. Since $v(z)$ decreases toward $0$ as $z \to 1$ (compare
 \refdef{HKFunctions}), it follows that $\frac{14.77}{L(s')^2}$ is bounded below, hence $L(s')$ is bounded above. Thus only finitely many slopes $s'$ can satisfy \refitm{S2short}.

If $N$ is a multi-cusped hyperbolic manifold, then the statement of \refthm{FinitenessOfPairs} still holds with slopes $s,s'$ replaced by tuples of slopes $\ss,\ss'$, as in \refthm{CosmeticImproved}. However, there can be infinitely many tuples satisfying the given upper bounds on $L(\ss)$ and $L(\ss')$.

We also have a version of \refthm{FinitenessOfPairs} for Dehn fillings of distinct manifolds.

\begin{theorem}\label{Thm:CommonFillingsPairs}
Let $N_1$,$N_2$ be one-cusped hyperbolic manifolds. Suppose that $s_1, s_2$ are slopes on $\bdy N_1, \bdy N_2$ respectively such that $N_1(s_1)$ is homeomorphic to $N_2(s_2)$. Suppose further that there is no homeomorphism of pairs $(N_1,s_1) \to (N_2, s_2)$. 
Then the following holds for the set $\{i,j\} = \{1,2\}$:
\begin{enumerate}[ \: \: $(1)$ ]
\item\label{Itm:S1shortCommon} $ L(s_i) < \max \left \{ 9.97, \, \sqrt{ \frac{2\pi}{\systole(N_i)}  + 56 } \right \} $.
\item\label{Itm:S2shortCommon} Either $L(s_j) < 9.97$ \: or \: $v \! \left( 1- \frac{14.77}{L(s_j)^2} \right) > \vol(N_j) - \vol(N_i(s_i))$.
\end{enumerate}
Here, $v(z)$ is the decreasing function of \refdef{HKFunctions}, and we are using the convention that the volume of a non-hyperbolic manifold is $0$.
\end{theorem}

\begin{proof}
Let $M$ be the manifold homeomorphic to both $N_1(s_1)$ and $N_2(s_2)$. 
Suppose that neither $s_1$ nor $s_2$ satisfy \refitm{S1shortCommon}. Then \refthm{UniqueShortest} says that $M$ is hyperbolic and both Dehn filling cores are isotopic to the unique shortest geodesic. Thus by Mostow rigidity,
there is an isometry taking the Dehn filling core of $N_1(s_1)$ to the Dehn filling core of $N_2(s_2)$, and taking meridian to meridian. This contradicts our assumption that there is no homeomorphism of pairs $(N_1,s_1) \to (N_2,s_2)$. Up to relabeling, we may assume that conclusion \refitm{S1shortCommon} holds for $i=1$.

Now, consider $s_j = s_2$. If $L(s_2) < 9.97$, then conclusion \refitm{S2shortCommon} holds. Alternately, if $L(s_2) \geq 9.97$, then \refthm{HKSimplerBound} implies $M \cong N_2(s_2)$ is hyperbolic. \refthm{HKSimplerBound} also implies that
\[
\vol(N_2) - \vol(M) \: < \: v \! \left(  1 - \frac{14.77}{L(s_2)^2}   \right) ,
\]
hence conclusion \refitm{S2shortCommon} holds for $j=2$. 
\end{proof}

\section{Procedure and results using hyperbolic tools}\label{Sec:CosmeticProcedure}

This section describes a procedure that tests Conjectures~\ref{Conj:Cosmetic} and~\ref{Conj:ChiralKnotCosmetic} on a given hyperbolic manifold. This procedure has been implemented in a Sage program~\cite{FPS:GitHubCosmetic}.
We use this program to prove \refthm{CosmeticCensus}, which verifies the cosmetic surgery conjecture, \refconj{Cosmetic}, on the SnapPy census up to 9 tetrahedra. We also use this procedure to prove \refthm{ChiralKnots}, which verifies the chirally cosmetic surgery conjecture, \refconj{ChiralKnotCosmetic}, on the set of hyperbolic knots up to 15 crossings.

\subsection{The procedure}\label{Sec:ProcedureDescription}

Here, we describe a computer program that tests a one-cusped hyperbolic manifold for cosmetic surgeries. As input, the program accepts a triangulation of a one-cusped hyperbolic 3-manifold $N$, or the SnapPy name of such a manifold. An optional flag specifies whether we wish to test for \emph{all} cosmetic surgeries or just purely cosmetic ones. As output, the program produces a list of pairs slopes that could not be rigorously distinguished.

The procedure implemented in our program is not an algorithm, in the sense that the program is not guaranteed to succeed. Nonetheless, any answer produced by the program is trustworthy. All SnapPy and Regina routines used in the program are rigorous, including rigorously verified intervals for hyperbolic invariants.

\subsubsection{Step 1. Gather basic information about $N$.}\label{Step:BasicInfo}

We begin by fixing a maximal horospherical neighborhood of the cusp, denoted $C$. Then, install a geometric framing $\langle \mu, \lambda \rangle$ for $H_1(C)$. That is, let $\mu$ be a shortest curve on $\bdy C$, and let $\lambda$ be a shortest curve among those that have intersection number $1$ with $\mu$.
Compute the cusp invariants: $\len(\mu)$, $\len(\lambda)$, and $\area(\bdy C)$. These computations can be rigorously certified.

Next, compute the homological longitude on $\bdy C$: that is, a primitive curve $\lamhom$  with $[\lamhom] = 0  \in H_1(N, \QQ)$. As mentioned in \refdef{HomologicalLongitude}, this curve is unique up to sign. 

Next, decide whether $N$ is amphicheiral. This computation can be rigorously certified using a positively oriented ideal triangulation. If $N$ has an orientation-reversing symmetry, then slopes $s$ and $-s$ (in the geometric framing) will always be chirally cosmetic. In this case, we test for purely cosmetic surgeries only.

Finally, if $H_1(N, \ZZ) \cong \ZZ$, apply \reftest{Casson} (Casson invariant) to $N$. In particular, if $H_1(N, \ZZ) \cong \ZZ$ and $\Delta''_N(1) \neq 0$, then $N$ cannot admit any purely cosmetic surgeries.

\smallskip

\subsubsection{Step 2. Find the exceptional fillings of $N$.}\label{Step:FindExcep}
By the $6$--theorem of Agol and Lackenby~\cite{agol:6theorem, lackenby:surgery}, any Dehn filling slope $s$ satisfying $\ell(s) > 6$ must produce a hyperbolic manifold $N(s)$. Consequently, all exceptional fillings must come from slopes with $\ell(s) \leq 6$. Using interval arithmetic, SnapPy can compute a set of slopes guaranteed to contain all those of length at most $6$.
For each such slope $s$, we attempt to verify that $N(s)$ is hyperbolic by finding a rigorously certified solution to the gluing equations. Failing that, we attempt to rigorously certify that $N(s)$ is exceptional (non-hyperbolic), in one of three ways:  by finding obstructions to hyperbolicity in the fundamental group; by finding an incompressible sphere or torus using normal surface theory; or by identifying $N(s)$ as a known Seifert fibered space using Regina. 

At the end of this step, the procedure produces a set $\calE = \calE(N)$ containing certified \emph{exceptional} slopes and a set $\calU = \calU(N)$ containing \emph{unidentified} slopes (whose hyperbolicity could not be settled). The set $\calU$ is reported but is not analyzed any further. However, in every instance where we ran our program, we have obtained $\calU = \emptyset$. This extends Dunfield's identification of $\calE(N)$ for every $N$ in the SnapPy census \cite[Theorem 1.2]{Dunfield:CensusFillings}, with the byproduct that $\calU(N) = \emptyset$. Indeed, our code incorporates and extends Dunfield's code.

\smallskip

\subsubsection{Step 3. Compare non-hyperbolic slopes.}\label{Step:Exceptional}
For every $s \in \calE$, we begin by trying to identify the closed manifold $N(s)$ using a combination of routines from Regina and SnapPy. This means performing the sphere and torus decompositions using Regina, identifying each hyperbolic piece using SnapPy, and identifying each Seifert piece using Regina's database of closed manifolds. This identification also uses Dunfield's code \cite{Dunfield:CensusFillings}. As a result, each non-hyperbolic $N(s)$ is assigned a \emph{Regina name} that describes the pieces. To save computation time, all of this information is computed just once for each $s \in \calE$.

Next, given $s \in \calE$ and $t \in \calE - \{s\}$, we try to distinguish $N(s)$ from $N(t)$. As a first test, we
compare the homology groups $H_1 N(s)$ and $H_1 N(t)$. Next, we use Regina names and the geometric decomposition. If only one of $\{N(s), N(t)\}$ is reducible, or only one is toroidal, then they are automatically distinguished. If both are irreducible and atoroidal, they must be Seifert fibered, and we compare their Seifert invariants. Otherwise, we compare the Seifert fibered pieces of the decomposition using Seifert invariants and hyperbolic pieces using hyperbolic invariants.

Our tests only analyze the pieces of the decomposition, but not the gluing data. Thus the above tests will fail when the pieces are identical. In addition, the tests can fail if we are unable to identify the Regina name of $N(s)$ or $N(t)$. (This has not occurred in our search, but is certainly possible on large examples.) 

As a final useful test, the program computes the homology of finite-index covers of $N(s)$ and $N(t)$, as well as of the corresponding normal cores. In practice, we compute the invariants $\calC(N(s), n)$ up to degree $n=7$ and $\calD(N(s), n)$ up to degree $n=6$, and similarly for $N(t)$. See Definitions~\ref{Def:HomologyOfCovers} and~\ref{Def:HomologyInvt}.
These comparisons of finite-index covers appear to distinguish toroidal manifolds particularly well.

\smallskip

\subsubsection{Step 4. Compute a set of systole--short slopes.}\label{Step:SystoleShort}
Now, we return to hyperbolic geometry. We begin by computing $\systole(N)$.  This can be rigorously certified in SnapPy; see \cite{GHHT:LengthSpectra}.

Next, we compute the set $\calS$ of all slopes that are shorter than the bound of \refthm{CosmeticImproved}:
\begin{equation}\label{Eqn:SDefine}
\calS = \left\{ s \text{ a slope} \: \bigg\vert \: L(s) < \max \left ( 9.97, \, \sqrt{ \tfrac{2\pi}{\systole(N)}  + 56 } \right ) \right\}.
\end{equation}

Next, we break $\calS$ into subsets (called \emph{buckets}) by homological data. For each number $p \in \NN \cup\{0\}$, let $\calS^p$ be the set
\[ \calS^p = \{s\in \calS : i(s, \lamhom) = \pm p \}, \]
where $i(\cdot, \cdot)$ denotes algebraic intersection number. By \reflem{Homology}, if $s \in \calS^p$ is part of a cosmetic surgery pair with an arbitrary slope $s'$, then  $i(s', \lamhom) = \pm p$ as well. Thus it makes sense to filter $\calS$ by intersection numbers with $\lamhom$.

By \refcor{HomLongitude}, $\calS^0$ is either empty or the singleton $\{ \lamhom \}$. Thus we may remove $\calS^0$ from consideration, and only consider $\calS^p$ for $p > 0$. For each $p \in \NN$, we define
 \[ \calH^p = \calS^p \smallsetminus (\calE \cup \calU). \]
 Recall that all non-hyperbolic slopes for $N$ must lie in $\calE \cup \calU$, hence $\calH^p$ consists of certified hyperbolic fillings.

\smallskip

\subsubsection{Step 5. Build comparison set of hyperbolic slopes.}
For each $p > 0$ such that $\calH^p \neq \emptyset$, define $V^p = \max\{\vol(N(s)) : s\in\calH^p \}$. Let $v(z)$ be the function of \refdef{HKFunctions}. Then, compute the following set of slopes:
\begin{equation}\label{Eqn:TDefine}
\qquad \calT^p = \left\{ t \text{ a slope} \: \bigg\vert \: i(t, \lamhom)=\pm p, \:\: t\notin\calE,  \:\: 
\left[ L(t) < 9.97  \: \text{ or } \: v \! \left( 1- \tfrac{14.77}{L(t)^2} \right) > \vol(N) - V^p \right] \right\}.
\end{equation}
In words, a slope $t \in \calT^p$ must satisfy three conditions. First,  $i(t, \lamhom)=\pm p$, hence $t$ is a potential cosmetic surgery partner for some $s \in \calH^p$. Second, $t\notin\calE$, hence we expect $N(t)$ to be hyperbolic. Third, $L(t)$ satisfies an upper bound on length calculated using the function
$v(z)$ of \refdef{HKFunctions}.

By \refthm{FinitenessOfPairs}, if $s,t$ are a cosmetic surgery pair and $s \in \calH^p$, then we must have $t \in \calT^p$. Thus it suffices to compare hyperbolic pairs $(s,t) \in \calH^p \times \calT^p$.

\smallskip

\subsubsection{Step 6. Compare hyperbolic slopes.}\label{Step:CompareHyperbolic}
For every  $p > 0$ such that $\calH^p \neq \emptyset$, and every pair $(s,t) \in \calH^p \times \calT^p$ such that $s \neq t$, we try to distinguish $N(s)$ from $N(t)$ using (mostly hyperbolic) invariants. We proceed as follows. First, try to distinguish $N(s)$ and $N(t)$ using their volumes, which have been rigorously computed using interval arithmetic. Second, try to distinguish $N(s)$ and $N(t)$ using the Chern--Simons invariant, which can also be computed using interval arithmetic modulo $\pi^2 / 2$. (The Chern--Simons invariant is sensitive to orientation, so this can be particularly useful in the setting where $N(s)$ and $N(t)$ are known to be a chirally cosmetic pair and we are looking for purely cosmetic surgeries.) Third, try to distinguish $N(s)$ from $N(t)$ using their (verified) complex length spectra up to some cutoff. 
Fourth, compare the core homology invariants $\calD( N(s), n)$ and $\calD(N(t), n)$ for degree $n \leq 7$.
If any of the above tests succeeds, we have rigorously distinguished $N(s)$ from $N(t)$.

\smallskip

\subsubsection{Step 7. Report any failures.}
At the end of this process, the program reports any points of failure. 
First, report the set $\calU$ of unidentified slopes (where we could not determine hyperbolicity).
Second, report any unequal pairs $(s,t) \in \calE \times \calE$ that we have failed to distinguish. Third, report any unequal pairs $(s,t) \in \calH^p \times \calT^p$ that we have failed to distinguish.

\subsection{Results}
The following theorems describe our computational results for both purely and chirally cosmetic surgeries.

\begin{theorem}\label{Thm:CosmeticCensus}
The cosmetic surgery conjecture,   \refconj{Cosmetic}, holds
for the 59,107 one-cusped manifolds in the SnapPy census.
\end{theorem}

\begin{proof}
For each manifold $N$ in the census, the procedure runs to completion. Each time, the set of unidentified slopes is $\calU(N) = \emptyset$. (This is a consequence of Dunfield's prior work \cite[Theorem 1.2]{Dunfield:CensusFillings}.) The procedure successfully distinguishes all exceptional pairs in $\calE \times \calE$ and all hyperbolic pairs in  $\calH^p \times \calT^p$, apart from exactly 7 slope pairs that are described in detail below.  In each case, the two fillings of $N$ produce oppositely oriented copies of the same closed 3-manifold $M$, where $M$ is chiral. Thus these fillings do not contradict  \refconj{Cosmetic}, which concerns purely cosmetic fillings.

The procedure starts by comparing non-hyperbolic manifolds.
There were five non-hyperbolic slope pairs that the program could not distinguish.
The first one is:

\begin{itemize}
\item $N = \texttt{m172}$, with $s = (0, 1)$ and $t= (1, 1)$. This is a well-known example, discovered by Bleiler, Hodgson, and Weeks \cite{BleilerHodgsonWeeks}. As described in detail in \cite[Section 3]{BleilerHodgsonWeeks}, the two fillings $N(s)$ and $N(t)$ are oppositely oriented copies of the lens space $L(49, 18)$. This lens space is chiral, because $49$ does not divide $18^2 + 1$.
\end{itemize}

For the next four slope pairs, the manifold $N$ is amphicheiral, with a symmetry interchanging slopes $s$ and $t$. Thus $N(s)$ has an orientation-reversing homeomorphism to $N(t)$. To check that the slopes $s,t$ satisfy \refconj{Cosmetic}, it suffices to show that $N(s)$ is a chiral manifold.
\begin{itemize}
\item $N = \texttt{m207}$ with $s = (0, 1)$ and $t =(1, 0)$. Here, $N(s)$ and $N(t)$ are oppositely oriented copies of $L(3,1) \# L(3,1)$. We check using Regina that each connected summand is $L(3,1)$ rather than the mirror image $\overline{L(3,1)}$. Since $L(3,1)$ is chiral, so is $L(3,1) \# L(3,1)$.

\item $N = \texttt{t12043}$ with $s = (1, -1)$ and $t = (1, 1)$. Here, $N(s)$ and $N(t)$ are graph manifolds with two Seifert fibered pieces. The first piece, $\texttt{SFS [D: (2,1) (2,1)]}$, is Seifert fibered over a disk with two singular fibers of slope $1/2$. The second piece,  $\texttt{SFS [M/n2: (2,1)]}$, is Seifert fibered over a M\"obius band with one singular fiber of slope $1/2$. Both of the Seifert pieces here are amphicheiral. However, the gluing matrix $m = 
\left[ \begin{smallmatrix} 0 & 1 \\ 1 & 0 \end{smallmatrix} \right]$ does not respect the symmetries of the pieces, so the resulting graph manifold is chiral.

\item $N = \texttt{t12045}$ with $s = (1, -1)$ and $t = (1, 1)$. Here, $N(s)$ and $N(t)$ are graph manifolds with three Seifert fibered pieces:  $\texttt{SFS [D: (2,1) (2,1)]}$ and $\texttt{SFS [A: (2,1)]}$ and $\texttt{SFS [D: (2,1) (2,1)]}$. The middle piece is Seifert fibered over the annulus, with one singular fiber of slope $1/2$. As in the previous bullet, the Seifert pieces are amphicheiral, but the gluing matrices $m =  \left[\begin{smallmatrix} 1 & \text{-}2 \\ 0 & \text{-}1 \end{smallmatrix} \right]$  and $ n =  \left[\begin{smallmatrix} 1 & 2 \\ 1 & 1 \end{smallmatrix} \right] $ do not respect the symmetries of the pieces.

\item $N = \texttt{t12050}$ with $s = (1, -1)$ and $t = (1, 1)$. Here, $N(s)$ and $N(t)$ are graph manifolds with two Seifert fibered pieces:  $\texttt{SFS [D: (2,1) (2,1)]}$ and $\texttt{SFS [D: (4,1) (4,1)]}$. As in the previous bullet, the Seifert pieces are amphicheiral, but the gluing matrix $m = 
\left[ \begin{smallmatrix} 0 & 1 \\ 1 & 0 \end{smallmatrix} \right]$ does not respect the symmetries of the pieces.
\end{itemize}

Finally, there were exactly two hyperbolic slope pairs that the program could not distinguish. In both cases, $N$ is amphicheiral with a symmetry interchanging $s$ and $t$. Thus $M = N(s)$ and $N(t)$ are isometric by an orientation-preserving isometry, and we need to show that $M$ is chiral.

\begin{itemize}
\item $N = \texttt{m135}$  with $s =(1, 3)$ and $t= (3, 1)$. Here, SnapPy was unable to compute verified hyperbolic invariants of $N(s)$, because it could not find any positively oriented spun triangulation. 

This verification problem is addressed by taking covers. The double covers of $N(s)$ and $N(t)$ do admit positively oriented triangulations. By a theorem of Hilden, Lozano, and Montesinos \cite[Theorem 2.2]{HildenLozanoMontesinos}, the Chern--Simons invariant 
 $\operatorname{CS}(M)$ is multiplicative under covers, modulo $\pi^2 /2$. The manifold $M = N(s)$ has six distinct double covers up to isomorphism. Each double cover $\widehat{M}$ satisfies $\operatorname{CS}(\widehat{M}) = - 0.8224... \mod \pi^2 /2$, which is verifiably distinct from $0$. 
 Thus $CS(M) = - 0.4112... \mod \pi^2 /4$, which is again verifiably distinct from $0$. Since $CS(N(s)) \neq 0$, we know that $N(s)$ is chiral,
 hence cannot admit an orientation-preserving isometry to $N(t)$.
 
 \item $N = \texttt{t12051}$ with $s = (1, 1)$ and $t = (2, -1)$.  Here, SnapPy successfully computed a verified hyperbolic structure on $M = N(s)$ from a positively oriented spun triangulation. The Chern--Simons invariant $CS(M)$ is indistinguishable from $0$, and does not prove that $M$ is chiral.
 
 Using this verified hyperbolic structure, SnapPy computed that the isometry group of $M$ is the dihedral group $D_8$. Every element of this group is orientation-preserving. Thus $M$ is chiral, hence $N(s)$ cannot admit an orientation-preserving isometry to $N(t)$.
\end{itemize}

Since each of the above slope pairs produce distinct oriented 3-manifolds, we conclude that \refconj{Cosmetic} holds for every $N$ in the census. 
\end{proof}

\begin{remark}\label{Rem:CosmeticCensusCode}
\refthm{CosmeticCensus} was proved by running the python program  \texttt{check\_mfds} on the 59,107 one-cusped manifolds in the SnapPy census.
A typical manifold in the census satisfies $\systole(N) > 0.145$, hence the bounds of \refthm{FinitenessOfPairs} yield a set $\calS = \calS(N)$ containing at most 102 short slopes. After filtering by homology, as in Steps 4 and 5 of the procedure, each short slope $s \in \calH^p$ needs to be compared to every slope $t \in \calT^p$, where $|\calT^p| \leq 5$ in a typical scenario. Thus Step 6 of our procedure needs to check no more than 500 slope pairs. For a typical manifold $N$ as above, the total runtime is several seconds. 

When $\systole(N)$ is very short, the program has to work somewhat harder, for two reasons. 
First, the systole itself is harder to compute. 
Our code performs this task by first drilling a very short curve $\gamma$, computing a hyperbolic structure on $N - \gamma$, and then re-filling $\gamma$; the full computation can take around 10 seconds. 
Second, the set of short slopes $\calS(N)$ computed in \refeqn{SDefine} has size approximately  $2\pi \cdot \systole(N)^{-1}$.
After filtering the hyperbolic slopes by homology, as in Steps 4 and 5, the number of slope \emph{pairs} to check can be on the order of $\systole(N)^{-3/2}$. An extreme example is the manifold $N = \texttt{o9\_00637}$, with $\systole(N) \approx 0.001643$. For this $N$, the set $\calS(N)$ contained  3698 hyperbolic slopes, which were sorted into 124 homology buckets.
Altogether, the program had to compare and distinguish 66,842 hyperbolic slope pairs on $N$, which took approximately 65 seconds. 

Moving beyond this extreme example, the program encountered 294 one-cusped manifolds whose systole is shorter than  $0.004$. Before the innovation of drilling and filling described in the above paragraph, each of those systoles took over an hour to compute. In the latest run of our code, each of these systoles could be computed and rigorously verified in under 10 seconds.
\end{remark}

\begin{remark}\label{Rem:SystoleUnneeded}
The shortest systole in the one-cusped SnapPy census is achieved by $N = \texttt{o9\_00639}$, with $\systole(N) \approx 0.001502$. This manifold has nonzero Casson invariant, so  \reftest{Casson} rules out purely cosmetic surgeries without using hyperbolic methods.
\end{remark}

\begin{theorem}\label{Thm:ChiralKnots}
The chirally cosmetic surgery conjecture for knots in $S^3$,   \refconj{ChiralKnotCosmetic}, holds on the set of all hyperbolic knots up to 15 crossings.
  \end{theorem}

\begin{proof}
We ran the procedure of \refsec{ProcedureDescription} on the set of 313,230 prime knots up to 15 crossings. Torus knots and satellite knots were identified and excluded, leaving 313,209 hyperbolic knots. Amphicheiral hyperbolic knots were excluded next. For each hyperbolic knot complement $N$, we found that $\calU(N)$, the set of unidentified slopes, is empty.
Thus every slope was rigorously identified as hyperbolic or exceptional. The code ran to completion, successfully distinguishing all exceptional pairs in $\calE \times \calE$ and all hyperbolic pairs in $\calH^p \times \calT^p$ (see Steps 3 and 6 of the procedure).
%
\end{proof}

\begin{remark}\label{Rem:AlmostAmphicheiral}
\refthm{ChiralKnots} was proved by running the python program \texttt{check\_mfds\_chiral} on the set of  313,209 hyperbolic knots up to 15 crossings.
For a typical knot complement, the program took 10--15 seconds to compute and rigorously verify the systole. (This is longer than the runtimes in \refrem{CosmeticCensusCode} because even a simplified triangulation of a  typical 15--crossing knot complement often has over 25 tetrahedra.) Then, the program took another several seconds to form the set $\calS(N)$ and distinguish the corresponding fillings.

However, there is a small list of 14--crossing knots, for which chirally cosmetic surgeries are particularly hard to distinguish:

\begin{verbatim}
hard_knots = ["14a506", "14a680",  "14a12813", "14a12858", "14a13107", "14a13262", 
"14a14042",  "14a17268", "14a17533", "14n1309", "14n1641", "14n1644", "14n2164"] 
\end{verbatim}
These 13 knots were previously identified by
Stoimenow \cite{Stoimenow:AlmostAmphichiral} as being \emph{almost amphicheiral}: that is,  each $K$ in the list is a mutant of its mirror image. Thus, for $N = S^3 - K$, the Dehn filling $N(p,q)$ is a genus--$2$ mutant of the mirror image of $N(p,-q)$. Since complex hyperbolic volume is preserved under mutation~\cite[Theorem 1.1 and Corollary 1.2]{MeyerhoffRuberman:MutationEta}, it cannot distinguish $N(p,q)$ from $N(p,-q)$. 

In practice, we distinguished these fillings using two invariants: rigorously certified length spectra up to length $3.5$, and the first homology of covers up to degree $7$ (see  \refdef{HomologyOfCovers}).

Either of these invariants suffices to distinguish $N(p,q)$ from $N(p,-q)$ for $N$ in the above set. Covers up to degree $5$ are very fast to compute, taking a few seconds at most. This suffices for most but not all pairs $s = (p,q)$ and $t = (p,-q)$. Covers of larger degree often take hours to compute. In an extreme, the computation of enough cover homology groups to distinguish the  Dehn fillings of  \texttt{14a12813} took over 6 days of wall time making full use of 8 cores, for a total of 48 days and 9 hours of CPU time.

By contrast, the rigorous computation of length spectra up to length $3.0$ or $3.5$ typically takes minutes or tens of minutes per pair of slopes. We needed to search up to this fairly large length because mutant manifolds share the same set of short curves up to some cutoff; compare Futer and Millichap \cite{FuterMillichap}. As a result, distinguishing all the Dehn fillings $N(p,q)$ and $N(p,-q)$ for one of the above manifolds often took several hours. For  \texttt{14a12813}, this method took 10 hours.
\end{remark}

\section{Common fillings of different parent manifolds}\label{Sec:CommonFillings}

In this section, we describe how the procedure of \refsec{CosmeticProcedure} can be adapted to determine common Dehn fillings of two distinct manifolds.
The results of this section were motivated by a question from Kalfagianni and Melby. In their joint work~\cite{KalfagianniMelby}, they use  our results (specifically \refthm{CommonFillingCensus}) to find many new $q$--hyperbolic knots in $S^3$.

\subsection{The procedure}\label{Sec:CommonFillProcedure}
Here, we describe a computer program that tests whether a pair of cusped hyperbolic manifolds $N_1$ and $N_2$ has any common Dehn fillings. The program takes $N_1$ and $N_2$ (or their names) as input, and produces as output a list of pairs of slopes $s_1$ on $N_1$ and $s_2$ on $N_2$. For each pair, the manifolds $N_1(s_1)$ and $N_2(s_2)$ are either certified as isometric, or are reported as undistinguished (hence, potentially homeomorphic). As with \refsec{ProcedureDescription}, our program is not guaranteed to succeed. However, as with that section, all computations are rigorous.

As the procedure described here is quite similar to that of \refsec{ProcedureDescription}, the description is somewhat more terse apart from Step 5, where new ideas are needed.

\subsubsection{Step 1. Gather basic information about $N_1$ and $N_2$}
This step is analogous to Step~1 in \S\ref{Step:BasicInfo}, applied to both $N_1$ and $N_2$. We compute the homological longitude for both $N_1$ and $N_2$. Next, we check that $N_1$ and $N_2$ are hyperbolic, but not isometric. If $N_1$ and $N_2$ are isometric, the program terminates because all Dehn fillings will be shared.

\subsubsection{Step 2. Find exceptional Dehn fillings of $N_1$ and $N_2$} This is analogous to Step~2 in \S\ref{Step:FindExcep}.
For all slopes of length $\ell(s) \leq 6$, certified by interval arithmetic, we check whether the Dehn filling is hyperbolic or not. For each parent manifold $N_i$, we obtain a set $\calE(N_i)$ containing certified exceptional slopes, and a set $\calU(N_i)$ containing unidentified slopes, meaning those for which the hyperbolicity could not be determined by the computer. In practice, $\calU(N_i)$ was empty for all manifolds we ran through the procedure. However, if $\calU$ is not empty, all slopes in $\calU(N_1)$ and all slopes in $\calU(N_2)$ will be reported, since the procedure could not check them for homeomorphisms. 

\subsubsection{Step 3. Compare non-hyperbolic Dehn fillings of $N_1$ and $N_2$}
For each slope $s_1 \in \calE(N_1)$, and each slope $s_2\in \calE(N_2)$, compare $N_1(s_1)$ to $N_2(s_2)$. As in Step~3 in \S\ref{Step:Exceptional}, we attempt to identify closed manifolds $N_1(s_1)$ and $N_2(s_2)$ using tools from SnapPy and Regina, as well as Dunfield's code. To distinguish them, we consider first homology groups, then Regina names and geometric decompositions, and then the cover homology invariants up to degree $n=7$. If the manifolds are not distinguished, we report these slopes. 

\subsubsection{Step 4. Compute a set of systole--short slopes for $N_1$ and $N_2$}
This is closely analogous to Step~4 in \S\ref{Step:SystoleShort}. For each $N_i$, we compute the set of systole-short slopes $\calS(N_i)$, exactly as in \refeqn{SDefine}. We remove the exceptional and unidentified slopes, obtaining a list $\calH(N_i)$ containing certified hyperbolic slopes. We filter $\calH(N_i)$ by homology, but we no longer remove the homological longitude $\lamhom$ because it could potentially be a common filling. In addition, since \reflem{Homology} and its corollaries are not very useful for comparing Dehn fillings of two surgery parents, we keep track of the actual homology group of $N_i(s_i)$.

\subsubsection{Step 5. Compare each systole-short slope $s_1 \in \calH(N_1)$ to an appropriate set of volume--short fillings of $N_2$.}
\label{Sec:ComparisonSetCommonFill}
This single step replaces Steps~5 and~6 in \refsec{ProcedureDescription}, with some important differences. 
For each short slope $s_1 \in \calH(N_1)$, we begin by comparing $\vol(N_1(s_1))$ to the volume of the cusped manifold $N_2$ using interval arithmetic in SnapPy. We define $V_1 = \vol(N_1(s_1))$ and consider three cases:
\[
V_1 > \vol(N_2), \qquad V_1 \sim \vol(N_2), \qquad V_1 < \vol(N_2).
\]
In the first case, SnapPy can certify that $\vol(N_1(s_1)) > \vol(N_2)$. Then we already know that no fillings of $N_2$ can be homeomorphic to $N_1(s_1)$, and we proceed to the next slope. 

In the second case, SnapPy is unable to rigorously distinguish $V_1 = \vol(N_1(s_1))$ from $\vol(N_2)$. If we had a rigorous proof that $\vol(N_1(s_1)) = \vol(N_2)$, then as in the first case we would conclude that no fillings of $N_2$ can be homeomorphic to $N_1(s_1)$. (Volume coincidences of this form often arise in the setting of arithmetic manifolds. See \refprop{VolumeV3} or a specific instance, and compare \cite[Theorem 11.2.3]{MaclachlanReid} for a more general statement.) Without a positive lower bound for $\vol(N_2) - V_1$, we cannot use \refthm{CommonFillingsPairs} to compute a set of comparison slopes. All we can do in this case is report the approximate equality $\vol(N_1(s_1)) \sim \vol(N_2)$, and proceed to the next slope. 

In the third case, SnapPy can certify that $V_1 < \vol(N_2)$. Then we build a comparison set of slopes on $N_2$ whose volumes might potentially be equal to $V_1 = \vol(N_1(s_1))$. 
Let $\calV(V_1, N_2)$ be the set of slopes short enough to satisfy the upper bound expressed in \refthm{CommonFillingsPairs}\refitm{S2shortCommon}:
\[
\calV(V_1, N_2) = \left\{ s_2 \text{ a slope on $\bdy N_2$} \: \bigg\vert \: L(s_2) < 9.97  \: \text{ or } \: v \! \left( 1- \tfrac{14.77}{L(s_2)^2} \right) > \vol(N_2) - V_1 \right\}.
\]
Then, in parallel with \refeqn{TDefine},
let
\[
\calT(s_1) =  \left 
 \{  s_2 \text{ a slope on $\bdy N_2$} \: \bigg\vert H_1 (N_2(s_2)) \cong H_1 (N_1(s_1)), \:\: s_2 \notin\calE,  \:\: 
  s_2 \in \calV(V_1, N_2)
  \right \}.
\]
In words, a slope $s_2 \in \calT(s_1)$ must satisfy three conditions. First, $N_2(s_2)$ must have appropriate first homology. (By \reflem{Homology}, all such slopes lie on a line in the Dehn surgery space of $N_2$.) Second, $s_2$ cannot be a known exceptional slope. Third, the normalized length $L(s_2)$ must satisfy the upper bound of \refthm{CommonFillingsPairs}\refitm{S2shortCommon}. By \refthm{CommonFillingsPairs}, any hyperbolic filling of $N_2$ that is homeomorphic to $N_1(s_1)$ must satisfy these constraints.

Continuing with the third case, we attempt to distinguish $N_1(s_1)$ from each manifold $N_2(s_2)$, as $s_2$ ranges over the set $\calT(s_1)$. This proceeds exactly as in \S\ref{Step:CompareHyperbolic}. We first compare hyperbolic volumes and Chern--Simons invariants, and then the homology groups of covers.

If the program cannot distinguish $N_1(s_1)$ from $N_2(s_2)$, then it attempts to prove they are isometric using SnapPy's (rigorous) isometry checker. In practice, for every $s_2 \in \calT(s_1)$, we have succeeded in either distinguishing the fillings or proving they are isometric.

\subsubsection{Step 6. Repeat Step 5, with the roles of $N_1$ and $N_2$ interchanged.} Observe that the manifolds $N_1$ and $N_2$ play very different roles in Step 5. Thus, to find all common hyperbolic fillings, we also need to compare each systole-short slope $s_2 \in \calH(N_2)$ to the appropriate comparison set of volume-short slopes on $\bdy N_1$.

It may well happen that there is a common filling $M = N_1(s_1) = N_2(s_2)$ where the slopes $s_1 \in \calH(N_1)$ and $s_2 \in \calH(N_2)$ are both systole-short on their respective manifolds. In this case, we will have found the pair $(s_1, s_2)$ twice. Thus, at the end of this step, we compare results to the output of Step 5 in order to remove duplicates.

\subsubsection{Step 7. Report the findings}
We report the sets $\calU(N_i)$ of unidentified slopes on the cusps of the manifolds. We report any slopes $s_i \in \calH(N_i)$ where the volume of $N_i(s_i)$ could not be distinguished from the volume of the other cusped manifold $N_j$. Finally, we report all slope pairs $(s_1, s_2)$ where the fillings were confirmed to be isometric and all slope pairs where the fillings could not be distinguished.

\subsection{Results}
We obtained the following computational results on common fillings of knot complements.

\begin{theorem}\label{Thm:CommonFillingLowCrossing}
Among the 12,955 hyperbolic knots with at most 13 crossings, the following 23 knots are a complete list of knots that share a nontrivial Dehn surgery with the figure-8 knot. 
All of the common fillings are hyperbolic.

\begin{verbatim}
common_fill = ["5_2", "6_1", "6_2", "7_2", "7_3", "8_1", "8_2", "8_20",  "9_2",
"9_3", "10_1", "10_2", "10_128", "10_132", "11a247", "11a364", "11n38", "11n57",
"12a722", "12a803", "12n243", "13a3143", "13a4874"]
\end{verbatim}

\end{theorem}

\begin{proof}
We ran the procedure of \refsec{CommonFillProcedure} on the set of 12,955 hyperbolic knots with at most 13 crossings. For each $K$ in the set, there were no unidentified slopes and no common exceptional fillings between the figure-8 knot and $K$ (apart from the trivial filling that produces $S^3$). This leaves hyperbolic fillings. For each hyperbolic filling of each $K$, apart from two examples ($Q =\texttt{5\_2(8,1)}$ and $R = \texttt{12n242(56,3)}$), our program was able to rigorously decide whether or not the resulting manifold also occurred as a filling of the figure--8 knot.

For the remaining two fillings, namely $Q =\texttt{5\_2(8,1)}$ and $R = \texttt{12n242(56,3)}$, SnapPy could not distinguish their volumes from that of the figure-8 knot complement $N = \texttt{4\_1}$. Indeed, the volumes could not be distinguished because they are equal. By 
 \refprop{VolumeV3}, we have an equality of volumes
\[
\vol(Q) = \vol(R)  = 2V_3 = \vol(N).
\]
Since every hyperbolic Dehn filling of $N$ has volume strictly less than $2V_3$, it follows that $Q$ and $R$ do not arise as fillings of the figure-8 knot.
\end{proof}

\begin{theorem}\label{Thm:CommonFillingCensus}
Among the 1267 hyperbolic knot complements that can be triangulated with fewer than 10 tetrahedra, there are  exactly 61 knots in the set that are distinct from the figure-8 knot and share a common nontrivial Dehn surgery with the figure-8 knot complement. Exactly one of these knots (namely, the $5_2$ knot) shares more than one nontrivial Dehn surgery with the figure-8 knot.  
\end{theorem}

\begin{proof}
This is identical to the proof of \refthm{CommonFillingLowCrossing}. The program successfully handles all fillings of every $K$ in the set, apart from the closed manifolds $Q$ and $R$ whose volume is identical to that of the figure--$8$ knot complement. As above,  \refprop{VolumeV3} shows that these manifolds cannot arise as Dehn fillings of the figure--$8$ knot complement.
\end{proof}

Tables of all the common Dehn surgeries can be found in Kalfagianni and Melby~\cite{KalfagianniMelby}. They are also reproduced in the ancillary files \cite{FPS:Ancillary}. 

\begin{remark}\label{Rem:CommonFillCode}
Theorems~\ref{Thm:CommonFillingLowCrossing} and \ref{Thm:CommonFillingCensus} are proved by running the program \texttt{find\_common\_fillings} on the data sets in the two theorems. For the average knot complement $N$ in the two data sets, the search for common fillings between $N$ and  the figure--$8$ knot complement takes approximately $5$ seconds.
\end{remark}

\section{Volumes of arithmetic manifolds}\label{Sec:Arithmetic}

The main goal of this section is to prove the following result, needed in  Theorems~\ref{Thm:CommonFillingLowCrossing} and \ref{Thm:CommonFillingCensus}.

\begin{proposition}\label{Prop:VolumeV3}
Let $Q$ be the closed 3-manifold produced by $(8,1)$ filling on the $5_2$ knot. Let $R$ be the closed 3-manifold produced by $(56,3)$ Dehn filling on the $(-2,3,7)$ pretzel knot. Then both $Q$ and $R$ are hyperbolic and arithmetic. Furthermore,
\[
\vol(Q) = \vol(R) = 2 V_3.
\]
\end{proposition}

The proof of  \refprop{VolumeV3} uses several deep but standard results about arithmetic hyperbolic 3-manifolds. In the next subsection, we give a brief summary of the definitions necessary to formulate results about volume; then, we will apply those arithmetic results to prove  \refprop{VolumeV3}.
We refer the reader to Maclachlan and Reid~\cite{MaclachlanReid} for a very thorough treatment of this material. A quick and accessible survey of arithmetic notions appears in Coulson, Goodman, Hodgson, and Neumann~\cite[Section 4]{CGHN:Snap}.

The reader who prefers to avoid technical arithmetic arguments is invited to skip the rest of this section, and consult \refrem{Clift} instead.

\begin{remark}\label{Rem:Clift}
In personal correspondence, Craig Hodgson and his student James Clift suggested an alternative method for verifying the equality
\[
\vol(Q) = \vol(R) = 2 V_3
\]
using geometric triangulations. 
Each of the two knot complements before the Dehn filling can be triangulated with three ideal tetrahedra. After Dehn filling, $Q = \texttt{5\_2(8,1)}$ has a spun ideal triangulation consisting of the same three tetrahedra, with the tips of the tetrahedra spinning about the core of the surgery solid torus. The same statement is true of $R = \texttt{12n242(56,3)}$. 

Clift has written code that can perform a sequence of $2$--$3$ moves, $3$--$2$ moves, and $2$--$0$ moves (see \cite[Section 2.39]{KSS:EssentialTriangulations} for definitions)
to verify that the three-tetrahedron spun triangulation of $Q$ is scissors congruent to the two-tetrahedron triangulation of the figure--$8$ knot complement. His code uses Snap~\cite{CGHN:Snap} to follow exact solutions to the gluing equations through the retriangulation moves, with shapes given in the number field $k(\Gamma) = \QQ(e^{\pi i /3})$. The same argument works for $R$. Since the manifolds are scissors congruent, they have equal volumes.
\end{remark}

The output of Clift's code is reproduced in the ancillary files \cite{FPS:Ancillary}. However, we have chosen to preserve the arithmetic proof of \refprop{VolumeV3} because the code is not publicly available.

\subsection{Arithmetic definitions}
A \emph{number field}, denoted $k$, is a finite extension of $\QQ$. The ring of integers of a number field $k$ is denoted $R_k$. A \emph{quaternion algebra} over $k$ is a $4$--dimensional $k$--vector space with basis $1, i, j, ij$, with multiplication defined so that $i^2=a \in k$ and $j^2=b \in k$ and $ij = -ji$.

A \emph{place} of a number field $k$ is an equivalence class of valuations $v: k \to \RR$. A place is called \emph{real} if it arises from the absolute value $| \cdot |$ in an embedding $k \hookrightarrow \RR$, meaning $v(x) = |x|$. A place is called  \emph{complex} if it arises from the absolute value $| \cdot |$ in a non-real embedding $k \hookrightarrow \CC$. A place $v$ is called \emph{finite} or \emph{non-Archimedean} if it satisfies $v(x+y) \leq \max(v(x), v(y))$ for all $x,y \in k$. Every place is real, complex, or finite, and finite places correspond to prime ideals $\calP \subset R_k$ \cite[Theorem 0.6.6]{MaclachlanReid}.

A quaternion algebra $A$ over $k$ is \emph{ramified} at finitely many places of $k$ \cite[Definitions 2.7.1]{MaclachlanReid}.  In fact, $A$ is determined by the set of places over which it is ramified \cite[Theorem 2.7.5]{MaclachlanReid}. The \emph{finite ramification locus} $\text{Ram}_f(A)$ is the set of prime ideals $\calP \subset R_k$ corresponding to the finite places where $A$ is ramified.

Now, let $\Gamma$ be a Kleinian group of finite covolume. Let $\Gamma^{(2)}$ be the finite-index subgroup of $\Gamma$ generated by all the squares of elements in $\Gamma$. The \emph{commensurability class}  $\calC(\Gamma)$ is the set of Kleinian groups $\Gamma'$ such that $\Gamma$ and $\Gamma'$ share a finite-index subgroup.

The \emph{invariant trace field} of $\Gamma$ is $k(\Gamma) = \QQ(\operatorname{tr} \Gamma^{(2)})$. This is always a number field \cite[Theorem 3.1.2]{MaclachlanReid}. In addition, $k(\Gamma)$ is an invariant of $\calC(\Gamma)$ \cite[Theorem 3.3.4]{MaclachlanReid}. 

Let $g, h \in \Gamma^{(2)}$ be any pair of non-commuting loxodromic elements. We lift $g,h \in PSL(2,\CC)$ to matrices in $SL(2,\CC)$. The \emph{invariant quaternion algebra} $A(\Gamma)$ is the $4$--dimensional algebra over $k(\Gamma)$ with basis $I, g, h, gh$:
\[
A(\Gamma) = k(\Gamma)[I, g, h, gh].
\] 
That this is a quaternion algebra is established in \cite[Corollary 3.2.3]{MaclachlanReid}.
While the choice of basis elements $g,h$ is highly non-unique, the algebra $A(\Gamma)$ is also an invariant of the commensurability class $\calC(\Gamma)$ \cite[Corollary 3.3.5]{MaclachlanReid}. When $\Gamma$ is arithmetic (see definition below), the pair $k(\Gamma)$ and $A(\Gamma)$ form a complete invariant of the commensurability class $\calC(\Gamma)$.

A finite-covolume Kleinian group $\Gamma$ is called \emph{arithmetic} if the following three conditions hold:
\begin{enumerate}
\item The invariant trace field $k(\Gamma)$ has exactly one complex place.
\item The invariant quaternion algebra $A(\Gamma)$ is ramified at all real places of $k(\Gamma)$.
\item $\Gamma$ has integer traces: that is, $\operatorname{tr}(\gamma)$ is an algebraic integer for every $\gamma \in \Gamma$.
\end{enumerate}
The equivalence between arithmeticity and these conditions is \cite[Theorem 8.3.2]{MaclachlanReid}.

It is also possible to produce a Kleinian group directly from a number field $k$ and a quaternion algebra $A$ over $k$ satisfying conditions $(1)$--$(2)$. 
An \emph{order} in $A$ is a ring of integers  $\calO \subset A$ that contains $R_k$ and satisfies $k \calO = A$. An order is \emph{maximal} if it is maximal with respect to inclusion. The set of units in a maximal order $\calO$ is denoted $\calO^1$. Orders in a quaternion algebra play a similar role to rings of integers in a field. Now, the single non-real embedding $k \hookrightarrow \CC$ induces an embedding $\rho \from A \to \operatorname{Mat}(2,\CC)$ into the set of $2 \times 2$ matrices over $\CC$. The image $\rho(A)$ is typically indiscrete, but restricting the domain to $\calO^1$ produces a discrete, faithful representation $\rho \from \calO^1 \to SL(2,\CC)$. Projecting down to $PSL(2,\CC)$ produces a  \emph{Kleinian group derived from a quaternion algebra}. This group, denoted $\Gamma_{\calO^1} = P \rho(\calO^1)$, has integer traces and is therefore arithmetic.

\subsection{Volumes}
The following is a special case of a foundational theorem due to Borel \cite[Theorem 11.1.3]{MaclachlanReid}.

\begin{theorem}\label{Thm:BorelZagier}
Let $k = \QQ(e^{\pi i/3})$, 
let $A$ be a quaternion algebra over $k$, and let $\calO$ be a maximal order in $A$. Then the Kleinian group $\Gamma_{\calO^1} = P \rho(\calO^1)$ derived from $A$ satisfies
\[
\vol(\HH^3 / \Gamma_{\calO^1}) = \frac{V_3}{6} \prod_{\calP \in \operatorname{Ram}_f(A)} (N(\calP) - 1),
\]
where $V_3$ is the volume of a regular ideal tetrahedron. The product is taken over the prime ideals corresponding to finite places where $A$ is ramified. For each prime ideal $\calP$, the norm $N(\calP)$ is the order of the finite field $R_k / \calP$.
\end{theorem}

%
%
%

\begin{proof}
Since $k = \QQ(e^{\pi i/3}) = \QQ(\sqrt{-3})$ is a quadratic imaginary extension of $\QQ$, it has no real places and exactly one complex place. Any quaternion algebra $A$ over $k$ is ramified at the empty set of real places, so $k$ and $A$ satisfy conditions $(1)$--$(2)$.

To compute $\vol(\HH^3 / \Gamma_{\calO^1})$ using Borel's volume formula, we need the following information.
 The field $k$ has discriminant $\Delta_k = -3$ by \cite[Example 0.2.7]{MaclachlanReid}. 
The degree of the extension is  $[k:\QQ] = 2$. The Dedekind zeta function $\zeta_k(2)$ was computed by Zagier \cite[page 291]{Zagier:ZetaHyperbolicVolume} in terms of the Lobachevsky function $\Lambda$:
\[
\zeta_k(2) = \frac{2 \pi^2}{3^{3/2}} \Lambda \Big(\frac{\pi}{3}\Big) = \frac{4 \pi^2}{3^{3/2}} \cdot \frac{V_3}{6}.
\]
With this information, Borel's volume formula \cite[Theorem 11.1.3]{MaclachlanReid} gives
\begin{align*}
 \vol(\HH^3 / \Gamma_{\calO^1}) 
 \: = \: \frac{|\Delta_k|^{3/2} \zeta_k(2)}{(4 \pi^2)^{[k:\QQ]-1}} \prod_{\calP \in \operatorname{Ram}_f(A)} (N(\calP) - 1) 
   \: = \: \frac{3^{3/2} }{4 \pi^2} \cdot \frac{4 \pi^2}{3^{3/2}} \cdot \frac{V_3}{6} \prod_{\calP \in \operatorname{Ram}_f(A)} (N(\calP) - 1), 
\end{align*}
as claimed.
\end{proof}

We can now prove \refprop{VolumeV3}. Recall that the $5_2$ knot appears in the SnapPy census as \texttt{m015}.
The $(-2,3,7)$ pretzel appears in the knot tables as $\texttt{12n242}$ and in the SnapPy census as \texttt{m016}.

\begin{proof}[Proof of \refprop{VolumeV3}]
Using SnapPy inside Sage, we verify that $Q = \mathtt{5\_2(8,1)}$ and $R = \mathtt{12n242(56,3)}$ are hyperbolic and estimate their volumes to within rigorously bounded error. 
\begin{verbatim}
sage: Q = snappy.Manifold("5_2(8,1)")
sage: Qvol = Q.volume(verified=True) # verified volume interval
sage: Qvol.lower(), Qvol.upper()     # endpoints of interval
(2.02988321281918, 2.02988321281944)

sage: R = snappy.Manifold("12n242(56,3)")
sage: Rvol = R.volume(verified=True) # verified volume interval
sage: Rvol.lower(), Rvol.upper()     # endpoints of interval
(2.02988321281924, 2.02988321281937)
\end{verbatim}
Note that both volume intervals have diameter less than $10^{-12}$ and contain $2 V_3$.

Next, let $\Gamma = \Gamma_Q$ be the Kleinian group such that $Q = \HH^3 / \Gamma$. We use Snap to compute the invariant trace field $k(\Gamma)$ and the invariant quaternion algebra $A(\Gamma)$:
\begin{verbatim}
5_2(8,1):
Field minimum poly (root): x^2 - x + 1 (1)
Real ramification: []
Finite ramification: [[3, [1, 1]~, 2, 1, [1, 1]~], [13, [-4, 1]~, 1, 1, [3, 1]~]]
Integer traces/Arithmetic: 1/1
\end{verbatim}
Thus $k(\Gamma) = \QQ(\alpha)$ where $\alpha = e^{\pi i /3}$ is a root of $x^2 - x + 1$. This field has exactly one complex place (because the two roots are complex conjugates) and no real places. Thus conditions $(1)$--$(2)$ above are satisfied. Snap also checks that $\Gamma$ has integer traces, so $(3)$ is satisfied and $\Gamma$ is arithmetic.

Snap checks that the quaternion algebra $A(\Gamma)$ is ramified at two finite places, corresponding to prime ideals $\calP_1 = (1 + \alpha)$ and $\calP_2 = (-4 + \alpha)$. Their norms are $N(\calP_1) = 3$ and $N(\calP_2) = 13$. Thus, taking a maximal order $\calO \subset A(\Gamma)$ and the group of units $\calO^1 \subset \calO$, \refthm{BorelZagier} tells us that the Kleinian group $\Gamma_{\calO^1} = P \rho(\calO^1)$ derived from the quaternion algebra $A = A(\Gamma)$ satisfies
\begin{equation}\label{Eqn:O1Vol}
\vol(\HH^3 / \Gamma_{\calO^1}) = \frac{V_3}{6} (3-1)(13-1) = 4V_3.
\end{equation}
Note that  $\Gamma_{\calO^1}$ belongs to the commensurability class $\calC(\Gamma) = \calC(A)$ because they have the same invariant trace field and quaternion algebra.

Next, we estimate the smallest covolume of any lattice in $\calC(\Gamma)$. By \cite[Theorem 11.5.2]{MaclachlanReid}, the smallest covolume is realized by a unique (up to conjugation) Kleinian group $\Gamma_\calO = \Gamma_{\emptyset, \calO}$.

By \cite[Theorem 11.6.3]{MaclachlanReid}, we learn $\Gamma_{\calO^1}$ is a normal subgroup of $\Gamma_\calO$, with quotient an abelian $2$--group. Since the ring of integers of $k(\Gamma)$, namely $\ZZ[e^{\pi i /3}]$, is a principal ideal domain, $k(\Gamma)$ has class number $1$, hence \cite[Theorem 11.6.5 and Corollary 11.6.4]{MaclachlanReid} give
\[
[\Gamma_\calO: \Gamma_{\calO^1}] = 2^n
\qquad \text{where} \qquad
n \leq 1 + |\operatorname{Ram}_f(A)| = 3.
\]
Combining this fact with \eqref{Eqn:O1Vol} gives
\begin{equation}\label{Eqn:OminVol}
\vol(\HH^3 / \Gamma_{\calO}) = \frac{4V_3}{2^n} \in \NN \cdot \frac{V_3}{2}.
\end{equation}

Finally, observe that the only prime in $k(\Gamma)$ that divides $2$ is $(2)$ itself, and this prime ideal is not ramified in $A(\Gamma)$. Thus \cite[Theorem 11.5.2]{MaclachlanReid} says that every Kleinian group $\Gamma$ commensurable to $\Gamma_{\calO^1}$ must have volume
an integer multiple of $\frac{1}{2}\vol(\HH^3 / \Gamma_{\calO})$. Combining this with \eqref{Eqn:OminVol}
gives 
\[
\vol(Q) \: = \:  \vol(\HH^3 / \Gamma) \: \in \:  \NN \cdot \frac{\vol(\HH^3 / \Gamma_{\calO})}{2} \: \subset \: \NN \cdot \frac{V_3}{4}.
\]
Since we know $\vol(Q)$ differs from $2V_3$ by far less than $V_3/8$, it follows that $\vol(Q) = 2V_3$.

Next, let $\Gamma = \Gamma_R$ be the Kleinian group whose quotient is $R$. Snap computes:
\begin{verbatim}
12n242(56,3)
Field minimum poly (root): x^2 - x + 1 (1)
Real ramification: []
Finite ramification: [[3, [1, 1]~, 2, 1, [1, 1]~], [13, [3, 1]~, 1, 1, [-4, 1]~]]
Integer traces/Arithmetic: 1/1
\end{verbatim}
Thus $R$ also has the same invariant trace field $k(\Gamma) = \QQ(\alpha)$ where $\alpha = e^{\pi i /3}$, and is again arithmetic. The quaternion algebra has changed: $A(\Gamma)$ is ramified at two finite places, corresponding to prime ideals $\calP_1 = (1 + \alpha)$ and $\calP_2 = (3 + \alpha)$. The norms of these prime ideals are $N(\calP_1) = 3$ and $N(\calP_2) = 13$, exactly the same as for $Q$. Since the number of ramified finite places and the norms of the corresponding ideals have not changed, the rest of the calculation -- including equations \eqref{Eqn:O1Vol} and \eqref{Eqn:OminVol} -- goes through verbatim. We conclude that $\vol(R) \in \NN \cdot V_3 /4$, hence $\vol(R) = 2V_3$.
\end{proof}

\section{Additional questions}

We close this paper with some open questions prompted by this work.
We begin by noting that Conjectures~\ref{Conj:Cosmetic}, \ref{Conj:CosmeticMultiCusp}, and \ref{Conj:ChiralKnotCosmetic}, stated in the introduction, all remain open. The results of this paper provide evidence for each of these conjectures.

There are also a few further questions that arise directly from the work in this paper. 

First, we were able to prove Conjectures~\ref{Conj:Cosmetic} and~\ref{Conj:CosmeticMultiCusp}
 for the one-cusped manifolds in the SnapPy census by reducing the problem to checking finitely many pairs of slopes for each manifold. Unfortunately, our methods do not extend to manifolds with more than one cusp. This is because in the one-cusped case, a fixed hyperbolic 3-manifold $N$ has at most finitely many pairs of slopes $(s,s')$ satisfying the constraints of \refthm{FinitenessOfPairs}.
If $N$ has multiple cusps,  \refthm{FinitenessOfPairs} still constrains the slope tuples $(\ss, \ss')$,
but it is no longer the case that only finitely many tuples satisfy these bounds.

\begin{question}\label{Ques:MultiCuspFiniteness}
  Is there a finiteness result for tuples of slopes on multi-cusped hyperbolic manifolds?
\end{question}

Second, observe that \refthm{FinitenessOfPairs}  gives a bound on the lengths of slopes that have to be compared for the purpose of excluding cosmetic surgeries. A key ingredient in the proof of  \refthm{FinitenessOfPairs} is \refthm{UniqueShortest}, which guarantees that the core $\sigma$ of the Dehn filling solid torus is the unique shortest geodesic in the  filled manifold $N(s)$. We wonder: 

\begin{question}
How sharp is \refthm{UniqueShortest}? Might its conclusion possibly hold for far shorter slopes?
\end{question}

Third, observe that the procedure of \refsec{CosmeticProcedure} -- which we applied to study \refconj{CosmeticMultiCusp} outside the realm of knot complements in $S^3$ as well as \refconj{ChiralKnotCosmetic} for knots in $S^3$ -- is confined to the setting of cusped hyperbolic manifolds. 
 It may well be feasible and practical to extend this procedure to  manifolds $N$ with torus boundary that contain an essential torus $T$. One case that deserves particular study is when the JSJ decomposition of $N$ is of the form $N = N_0 \cup_T N_1$, where 
$N_0$ meets $\bdy N$ and happens to be hyperbolic. In this case, all Dehn fillings of $N$ will be of the form 
\[
N(s) = N_0(s) \cup_T N_1,
\]
where $N_0(s)$ is hyperbolic for almost all slopes $s$. Thus, for almost all slopes $s$, the torus $T$ remains incompressible, and the JSJ decomposition of $N(s)$ is exactly $N_0(s) \cup_T N_1$. So, to distinguish $N(s)$ from $N(t)$, it suffices to distinguish $N_0(s)$ from $N_0(t)$.
Here,  the multi-cusped hyperbolic manifold $N_0$ is being filled along a single slope on a single cusp, so the finiteness issues discussed in \refques{MultiCuspFiniteness} do not arise. Thus our techniques should readily adapt to this situation. However, this is not currently implemented.

\begin{question}
  Can the techniques of \refsec{CosmeticProcedure} be extended to 3-manifolds with an embedded essential torus? More ambitiously, can the techniques of \refsec{CommonFillings} be extended to this setting as well?
\end{question}

Finally, we observe that in \refthm{CommonFillingLowCrossing}, the number of $n$--crossing knots that share a common filling with the figure--$8$ knot does not seem to increase with $n$. By contrast, the total number of $n$--crossing knots grows exponentially. Hence, we wonder:

\begin{question}
For $n \geq 4$, let $\mathcal K_n$ be the set of hyperbolic knots of $n$ crossings that share a common filling with the figure--$8$ knot. 
By \cite[Theorem 1.2]{KalfagianniMelby}, every $\mathcal K_n$ contains at least one double twist knot, so it is nonempty. 
Is the cardinality $| \mathcal K_n |$ universally bounded as $n$ grows? The same question also makes sense if the figure--$8$ knot is replaced by any fixed $K$.
\end{question}

\bibliographystyle{amsplain}
\bibliography{biblio}

\end{document}